\documentclass{amsart}

\usepackage[foot]{amsaddr}

\usepackage{times}

\usepackage{bbm}
\usepackage{hyperref}
\usepackage{todonotes}
\usepackage{url}

\usepackage{paralist}

\usepackage{booktabs}
\usepackage{supertabular}
 
\DeclareMathOperator{\conv}{conv}
\DeclareMathOperator{\down}{down}
\DeclareMathOperator{\homog}{homog}
\DeclareMathOperator{\rec}{rec}
\DeclareMathOperator{\rc}{rc}
\DeclareMathOperator{\rrc}{rrc}
\DeclareMathOperator{\supp}{supp}
\DeclareMathOperator{\xc}{xc}
\DeclareMathOperator{\xcs}{xc*}

\newcommand{\cdd}{\texttt{cdd}}
\newcommand{\Hleq}{H^{\leq}}
\newcommand{\Heq}{H^=}
\newcommand{\numfacets}{\# \mathrm{facets}}
\newcommand{\numverts}{\# \mathrm{vertices}}
\newcommand{\polymake}{\texttt{polymake}}
\newcommand{\R}{\mathbb{R}}
\newcommand{\unitvec}{\mathbbm{e}}
\newcommand{\zerovec}{\mathbb{O}}

\newcommand{\flipmap}{\varphi_{\mathrm{flip},i}}
\newcommand{\swapmap}{\varphi_{\mathrm{swap},i,j}}

\newtheorem{theorem}{Theorem}[section]
\newtheorem{lemma}[theorem]{Lemma}
\newtheorem{proposition}[theorem]{Proposition}
\newtheorem{corollary}[theorem]{Corollary}
\newtheorem*{observation*}{Observation}

\theoremstyle{definition}
\newtheorem{definition}{Definition}

\theoremstyle{remark}
\newtheorem{remark}{Remark}

\title{Computing The Extension Complexities of All 4-Dimensional 0/1-Polytopes}
\author{Michael Oelze$^1$}
\address{$^1$Otto-von-Guericke-Universit\"at Magdeburg (Germany), \normalfont{\texttt{michael5oelze@gmail.com}}}
\author{Arnaud Vandaele$^2$}
\address{$^2$Universit\'e de Mons (Belgium), \normalfont{\texttt{arnaud.vandaele@umons.ac.be}}}
\author{Stefan Weltge$^3$}
\address{$^3$Otto-von-Guericke-Universit\"at Magdeburg (Germany), \normalfont{\texttt{weltge@ovgu.de}}}

\begin{document}

\begin{abstract}
    We present slight refinements of known general lower and upper bounds on sizes of extended formulations for
    polytopes.
    With these observations we are able to compute the extension complexities of all 0/1-polytopes up to dimension 4.
    We provide a complete list of our results including geometric constructions of minimum size extensions for all
    considered polytopes.
    Furthermore, we show that all of these extensions have strong properties.
    In particular, one of our computational results is that every 0/1-polytope up to dimension 4 has a minimum
    size extension that is also a 0/1-polytope.
\end{abstract}

\maketitle

\section{Introduction}

\noindent
The theory of extended formulations is a fast-developing research field that adresses the problem of writing a polytope
as the projection of a preferably simpler polyhedron.
More precisely, given a polytope $ P \in \R^p $, a polyhedron $ Q \in \R^q $ together with a linear map $ \pi \colon
\R^q \to \R^p $ is an \emph{extension} of $ P $ if $ \pi(Q) = P $.
An explicit outer description of $ Q $ by linear inequalities and equations is called an \emph{extended
formulation} for $ P $.
The \emph{size} of an extension ($ Q $, $ \pi $) is defined as the number of the facets of $ Q $.
The quantity of major interest in the field of extended formulations is the so-called \emph{extension complexity} of a
polytope $ P $, which is defined as the smallest size of any extension of $ P $ and is denoted by $ \xc(P) $.
Equivalently, $ \xc(P) $ is the smallest number of inequalities in any extended formulation for $ P $.

The concept of extended formulations is motivated by the following fact: Suppose that $ (Q,\pi) $ is an extension of  $
P $.
Then optimizing a linear function $ x \mapsto \langle c,x \rangle $ over $ P $ is equivalent to maximizing $ y \mapsto
\langle \pi^*(c),y \rangle $ over $ Q $, where $ \pi^* $ is the adjoint map of $ \pi $.
Of course, if $ Q $ admits a simpler outer description than $ P $, then this might have a substantial impact on the
performance of algorithms solving such optimization problems.
Since $ 0/1 $-polytopes (i.e., polytopes with vertices in $ \{0,1\}^p $) play a central role in the application of
linear programming, they are also of particiular interest in the field of extended formulations.

Current research is mainly driven by the seminal results of Fiorini et al.~\cite{FioriniMPTW12} and
Rothvo\ss{} \cite{Rothvoss13} that give exponential lower bounds on the extension complexities of the TSP polytope and
the matching polytope, respectively.
While they answered two of the most important questions in the area of extended formulations, there are still many
(sometimes even elementary) open questions.
By giving a complete list of the extension complexities of all $ 0/1 $-polytopes up to dimension $4$, one aim of this
paper is to provide a first reference that hopefully allows progress on those questions.

For instance, it is not known whether for any rational polytope $ P $ there exists a rational extension of size $
\xc(P) $, not even if $ P $ is a $ 0/1 $-polytope.
As a consequence of observations mainly made in Section~\ref{sec:upperbounds}, our computational results show that
all $ 0/1 $-polytopes up to dimension $ 4 $ admit minimum size extensions that are even $ 0/1 $-polytopes.
Surprisingly, in all those extensions $ Q $ there is a 1-to-1 correspondence between the vertices of $ Q $ and their
images.

Another motivation for this work was the fact that, in general, computing $ \xc(P) $ for a given polytope $ P $ seems to
be a non-trivial task.
Note that, from its original definition, it is not obvious how to compute the extension complexity of a polytope.
However, due to Yannakakis' theorem \cite{Yannakakis91}, we know that computing $ \xc(P) $ is equivalent to computing
the nonnegative rank of one of its slack matrices (see Section \ref{sec:lowerbounds} for the precise definitions and the
statement).
For reasonable numbers of facets and vertices, it is possible to compute slack matrices using the double description
method, which is, for instance, efficiently implemented in \cdd~\cite{Fukuda14} and used by
\polymake~\cite{GawrillowJ00}.
Further, computing the exact nonnegative rank of a matrix can be reduced to the decision problem whether certain
semi-algebraic sets are nonempty.
Arora et al.~\cite{AroraGKM12} give a subtle construction of such sets whose descriptions are much smaller than
in naive approaches.
Thus, in principle, it is possible to compute the extension complexity of a polytope by finally using quantifier
elimination algorithms as in~\cite{BasuPR96}.
However, such general approaches still do not allow computations for polytopes of reasonable complexity.
In contrast, our calculations are based on very specific observations yielding matching lower and upper bounds on $
\xc(P) $.
Hence, they can be performed within a total time of few minutes using simple scripts as well as \polymake{} for
computing slack matrices.

Regarding our specific task, note that the extension complexity is obviously invariant under affine transformations.
Thus, when talking about $ 0/1 $-polytopes of dimension $ k $, we implicitly refer to full-dimensional polytopes with
vertices in $ \{0,1\}^k $.
(It is a basic fact that any $ 0/1 $-polytope of dimension $ k $ is affinely isomorphic to a full-dimensional $ 0/1
$-polytope in ambient dimension $ k $.)
Formally, there are $ 2^{2^4} = 65536 $ polytopes with vertices in $ \{0,1\}^4 $ -- most of them being
$ 4 $-dimensional.
Of course, it suffices to consider only
representatives of each affine equivalence class.
It turns out that there are still $ 202 $ distinct affine equivalence classes of $ 0/1 $-polytopes of dimension $4$.
While we also give results for (the few) $ 0/1 $-polytopes up to dimension three, our work mainly focusses on the more
challenging class of $ 4 $-dimensional $ 0/1 $-polytopes.

Our paper is organized as follows.
In Section~\ref{sec:upperbounds}, we describe known, simple geometric constructions for extended formulations and show
that they preserve interesting properties.
In order to obtain tight bounds on $ \xc(P) $, we carefully analyze the sizes of the resulting extensions.
In Section~\ref{sec:lowerbounds}, we recall known general lower bounds on the extension complexity of a polytope and
present a refinement of the rectangle covering bound, which yields improved bounds but is still computable by rather
simple combinatorial algorithms.
Finally, in Section~\ref{sec:computations}, we describe our approach of computing all extensions complexities and
present computational results.

\subsection{Notation}
The standard euclidean scalar product and norm are denoted by $ \langle \cdot,\cdot \rangle $ and $ \| \cdot \| $,
respectively.
For a nonnegative integer $ k $, we set $ [k] := \{1,\dotsc,k\} $.
The $ k $-dimensional nonnegative orthant is denoted by $ \R^k_+ $.
We use $ \Delta_k := \{ x \in \R^k_+ : \sum_{i=1}^k x_i = 1 \} $ to denote the standard $ (k-1) $-simplex.
The zero vector will be denoted by $ \zerovec $, whose dimension will always be clear from the context.
For a polyhedron $ P $, we use $ \rec(P) $ to denote its recession cone.
Further, $ \numfacets(P) $ and $ \numverts(P) $ denote the number of facets of $ P $ and number of vertices of $ P $,
respectively.

\section{Geometric Upper Bounds}
\label{sec:upperbounds}
\noindent
In this section, we review known bounds on the extension complexity that are based on simple geometric constructions.
We show that some bounds can be sligthly strengthened in relevant cases, which will be essential for the computations in
Section~\ref{sec:computations}.
Besides, we show that in our cases, the corresponding constructions preserve the following, strong properties:
\begin{definition}
    \label{def:nice01}
    Let $ P $ be a $ 0/1 $-polytope and $ (Q,\pi) $ be an extension for $ P $.
    Then $ (Q,\pi) $ is called a \emph{nice $ 0/1 $-extension} if
    \begin{compactenum}[(i)]
        \item \label{enum:zeroone} $ Q $ is a $ 0/1 $-polytope,
        \item \label{enum:vertexpreserving} each vertex of $ Q $ is projected onto a vertex of $ P $ and
        \item \label{enum:vertexbijective} for each vertex $ v $ of $ P $ there is exactly one vertex of $ Q $ that
        projects onto $ v $.
    \end{compactenum}
    The smallest size of any nice $ 0/1 $-extension of $ P $ is denoted by $ \xcs(P) $.
\end{definition}
\noindent
Clearly, we have that $ \xc(P) \leq \xcs(P) $ holds for any $ 0/1 $-polytope $ P $.
In general, the requirements of the above definition seem to be very restrictive.
In fact, in~\cite{PashkovichW14} it was shown that there exist polytopes for which no minimum size extension satisfies
property~(\ref{enum:vertexpreserving}).
However, the polytopes constructed in that paper were not $ 0/1 $-polytopes.
In particular, we will see that all minimum size constructions that are implicitly generated by our computations are
indeed nice $ 0/1 $-extensions.
\begin{remark}
    Our convention of restricting projections to be linear instead of affine maps is just a technical requirement:
    Indeed, if $ P = \alpha(Q) $ for an affine map $ \alpha $, define $ Q' := \{ (x,y) : x = \alpha(y), \, y \in Q \} $,
    which is affinely isomorphic to $ Q $, let $ \pi $ be the linear projection onto the $ x $-coordinates and obtain
    that $ (Q',\pi) $ is an extension for $ P $.
    Moreover, if $ P $ is a $ 0/1 $-polytope and $ (Q,\alpha) $ satisfies the above-named properties
    (\ref{enum:zeroone})--(\ref{enum:vertexbijective}), then so does $ (Q',\pi) $.
    In particular, we still have that $ \xc(P) = \xc(P') $ and $ \xcs(P) = \xcs(P') $ holds for affinely isomorphic
    polytopes $ P, P' $.
\end{remark}
\noindent
Let us start with two trivial known upper bounds on the extension complexity of a polytope $ P \subseteq \R^p $.
First, choosing $ Q = P $ and $ \pi \colon \R^p \to \R^p $ as the identity, any polytope is an extension of itself.
Second, if $ P = \conv (\{ v^1, \dotsc, v^k \}) $ for some points $ v^1, \dotsc, v^k \in \R^p $, then we obviously have
that
\[
    P = \Big\{ \sum\nolimits_{i=1}^k \lambda_i v^i :
    \lambda_j \geq 0 \ \forall \, j \in [k], \,
    \sum\nolimits_{j=1}^k \lambda_j = 1 \Big\}.
\]
Thus, setting $ Q = \Delta_k $ and $ \pi(\lambda) = \sum_{i=1}^k \lambda_i v^i $, we obtain that $ P $ is a linear
projection of a $ (k-1) $-simplex.
Since both are clearly nice $ 0/1 $-extensions if $ P $ is a $ 0/1 $-polytope, we conclude:
\begin{proposition}
    \label{prop:trivialupperbounds}
    For a $ 0/1 $-polytope $ P $ it holds that
    \begin{itemize}
        \item $ \xcs(P) \leq \numfacets(P) $ and
        \item $ \xcs(P) \leq \numverts(P) $. \qed
    \end{itemize}
\end{proposition}
\noindent
In fact, for some $ 0/1 $-polytopes such trivial extensions are indeed best possible.
Let us now describe two more subtle ways to construct extended formulations.

\subsection{Unions of Polytopes}
For a polyhedron $ Q = \{ y : Ay \leq b \} \subseteq \R^q $ let us denote its homogenization cone by
\[
    \homog(Q) = \{ (y,\lambda) : Ay \leq \lambda b, \, \lambda \geq 0 \}
    = \{ (y,\lambda) : y \in \lambda \cdotp Q, \, \lambda \geq 0 \} \subseteq \R^{q+1}.
\]
Since Balas' famous work~\cite{Balas79} on Disjunctive Programming, we know that for polytopes $ P_1,\dotsc,P_k \in \R^p
$, the convex hull of their union $ P = \conv( \cup_{i=1}^k P_i ) $ can be described via
\[
    P = \Big\{ \sum\nolimits_{i=1}^k x^i : (x^i,\lambda_i) \in \homog(P_i) \ \forall \, i \in [k], \,
    \sum\nolimits_{i=1}^k \lambda_i = 1 \Big\}.
\]
Given extensions $ (Q_i,\pi_i) $ for each $ P_i $, as a direct consequence, we obtain
\begin{equation}
    \label{eq:disjunction}
    P = \Big\{ \sum\nolimits_{i=1}^k \pi_i(y^i) : (y^i,\lambda_i) \in \homog(Q_i) \ \forall \, i \in [k], \,
    \sum\nolimits_{i=1}^k \lambda_i = 1 \Big\}.
\end{equation}
Thus, choosing each extension $ (Q_i,\pi_i) $ of minimum size, this immediatly implies the well-known upper bound $
\xc(P) \leq \sum_{i=1}^k (\xc(P_i) + 1) $, see, e.g., \cite{Kaibel11a}.
However, we show that this bound can be slightly improved in most of the cases:
\begin{theorem}
    \label{thm:xcunion}
    For polytopes $ P_1, \dotsc, P_k \in \R^p $ it holds that
    \[
        \xc\big( \conv(\cup_{i=1}^k P_i)\big) \leq \sum_{i=1}^k \xc(P_i) + |\{i \in [k] : \dim(P_i) = 0 \}|.
    \]
\end{theorem}
\noindent
When proving Theorem~\ref{thm:xcunion}, we will make use of the following (known) useful fact.
\begin{lemma}
    \label{lem:polytope}
    For any polytope $ P $ there exists an extension $ (Q,\pi) $ of minimum size such that $ Q $ is a polytope.
\end{lemma}
\begin{proof}
    \renewcommand{\qedsymbol}{}     See Appendix~\ref{proof:polytope}.
\end{proof}
\begin{proof}[Proof of Theorem~\ref{thm:xcunion}]
    For $ i=1,\dotsc,k $ let $ (Q_i,\pi_i) $ be a minimum size extension for $ P_i $.
    By Lemma~\ref{lem:polytope}, we further may assume that all $ Q_i $'s are polytopes.
    By equation~\eqref{eq:disjunction}, it suffices to show that $ \numfacets(\homog(Q_i)) \leq \numfacets(Q_i) $ for
    all $ i \in [k] $ with $ \dim(P_i) > 0 $.

    Towards this end, suppose that $ \dim(P_j) > 0 $ and let $ Q_j = \{ y : Ay \leq b \} $.
    We will show that $ \homog(Q_j) = \{ (y,\lambda) : Ay \leq \lambda b, \, \lambda \in \R \} $ (and hence $
    \numfacets(\homog(Q_j)) \leq \numfacets(Q_j) $) holds.
    In order to show that the inequality $ \lambda \geq 0 $ is indeed not facet-defining for $ \homog(Q_j) $ (and hence
    redundant), let us assume the contrary and obtain
    \begin{align*}
        \dim(\homog(Q_j))
        & = \dim(\{(y,0) \in \homog(Q_j) \}) + 1 \\
        & = \dim(\{(y,0) : Ay \leq \zerovec \}) + 1 \\
        & = \dim(\rec(Q_j) \times \{0\}) + 1 \\
        & = \dim(\{\zerovec\} \times \{0\}) + 1 = 1.
    \end{align*}
    On the other hand, we have that $ \dim(\homog(Q_j)) = \dim(Q_j) + 1 \geq \dim(P_j) + 1 \geq 2 $, a contradiction.
\end{proof}
\noindent
It turns out that Theorem~\ref{thm:xcunion} can even be rephrased in terms of $ \xcs(\cdot) $ under a further
assumption:
\begin{theorem}
    \label{thm:xcsunion}
    For disjoint $ 0/1 $-polytopes $ P_1, \dotsc, P_k \in \R^p $ it holds that
    \[
        \xcs\big( \conv(\cup_{i=1}^k P_i)\big) \leq \sum_{i=1}^k \xcs(P_i) + |\{i : \dim(P_i) = 0 \}|.
    \]
\end{theorem}
\begin{proof}
    For $ i=1,\dotsc,k $ let $ (Q_i,\pi_i) $ be a nice $0/1$-extension.
    Since the $ Q_i $'s are polytopes and due to the proof of Theorem~\ref{thm:xcunion}, it suffices to show that the
    polytope
    \[
        Q = \Big\{ (y^1,\dotsc,y^k,w) : (y^i,w_i) \in \homog(Q_i) \ \forall \, i \in [k], \, \sum\nolimits_{i=1}^k
        w_i = 1 \Big\}.
    \]
    together with $ \pi(y^1,\dotsc,y^k,w) = \sum_{i=1}^k \pi_i(y^i) $ is nicely $ 0/1 $.
    Towards this end, let $ v = (y^1,\dotsc,y^k,w) $ be a vertex of $ Q $ and let us define
    \[
        x^j := \begin{cases} \frac{1}{w_j} y^j & \text{ if } w_j > 0, \\
        \zerovec & \text{ if } w_j = 0.
        \end{cases}
    \]
    Note that $ x^j \in Q_j $ if $ w_j > 0 $.
    Let us recall that $ (y^j,0) \in \homog(Q_j) $ if and only if $ y^j \in \rec(Q_j) $.
    Since $ Q_j $ is a polytope, we obtain that $ (y^j,0) \in \homog(Q_j) $ if and only if $ y^j = \zerovec $.
    Thus, we get
    \begin{equation}
        \label{eq:zeroonebalas1}
        w_j x^j = y^j
    \end{equation}
    as well as
    \begin{equation}
        \label{eq:zeroonebalas2}
        (\tilde{w}_j x^j, \tilde{w}_j) \in \homog(Q_j) \quad \forall \ j \in [k], \, \tilde{w}_j \geq 0.
    \end{equation}
    We claim that $ w $ is a vertex of $ \Delta_k $.
    If not, then there exist $ \overline{w}, \underline{w} \in \Delta_k $ with $ \overline{w} \neq \underline{w} $ such
    that $ w = \mu \cdotp \overline{w} + (1-\mu) \cdotp \underline{w} $ for some $ \mu \in (0,1) $.
    By equation~\eqref{eq:zeroonebalas2}, we have that $ (\overline{w}_j \cdotp x^j,\overline{w}_j), (\underline{w}_j
    \cdotp x^j,\underline{w}_j) \in \homog(Q_j) $ holds for all $ j \in [k] $ and hence
    \begin{align*}
        \overline{v} := (\overline{w}_1 \cdotp x^1, \dotsc, \overline{w}_k \cdotp x^k, \overline{w}) \in Q
        \quad \text{ and } \quad
        \underline{v} := (\underline{w}_1 \cdotp x^1, \dotsc, \underline{w}_k \cdotp x^k, \underline{w}) \in Q.
    \end{align*}
    On the other hand, we also have
    \[
        \mu \cdotp \overline{w}_j \cdotp x^j + (1-\mu) \cdotp \underline{w}_j \cdotp x^j = w_j x^j
        \stackrel{\eqref{eq:zeroonebalas1}}{=} y^j,
    \]
    a contradiction to $ v $ being a vertex of $ Q $.

    Since $ w $ is a vertex of $ \Delta_k $, by relabeling the indices, we may assume that $ w_1 = 1 $ and $ w_j = 0 $
    for all $ j \in \{2,\dotsc,k\} $.
    As mentioned above, this implies $ y^j = \zerovec $ for all $ j \in \{2,\dotsc,k\} $.
    We claim that $ y^1 $ is a vertex of $ Q_1 $.
    For the sake of a contradiction let us assume that there exist $ \overline{y}, \underline{y} \in Q_1 $ with $
    \overline{y} \neq \underline{y} $ such that $ y^1 = \mu \cdotp \overline{y} + (1-\mu) \cdotp \underline{y} $ for
    some $ \mu \in (0,1) $.
    Since $ w_1 = 1 $, we have that $ (\overline{y},w_1), (\underline{y},w_1) \in \homog(Q_1) $ and hence
    \begin{align*}
        \overline{v} := (\overline{y},\zerovec,\dotsc,\zerovec,w) \in Q
        \quad \text{ and } \quad
        \underline{v} := (\underline{y},\zerovec,\dotsc,\zerovec,w) \in Q.
    \end{align*}
    On the other hand, we also have $ v = \mu \cdotp \overline{v} + (1-\mu) \cdotp \underline{v} $, which again is a
    contradiction to $ v $ being a vertex of $ Q $.

    Since $ Q_1 $ is a $ 0/1 $-polytope, $ y^1 $ is also a $ 0/1 $-vector and so is $ v $.
    Moreover, since $ (Q_1,\pi_1) $ is a nice $ 0/1 $-extension, we obtain that $ \pi(v) = \pi_1(y^1) $ is a vertex of
    $ P_1 $ and hence a vertex of $ P := \conv\big(\cup_{i=1}^k P_i \big) $.
    Thus, $ Q $ is indeed a $ 0/1 $-polytope and each vertex of $ Q $ is projected onto a vertex of $ P $.

    In order to verify property~(\ref{enum:vertexbijective}), let $ v $ be a vertex of $ P $.
    Since the $ P_i $'s are $ 0/1 $-polytopes, there exists an index $ \ell \in [k] $ such that $ v $ is a vertex of $
    P_{\ell} $.
    Further, let $ v' = (y^1,\dotsc,y^k,w) $ be a vertex of $ Q $ such that $ \pi(v') = v $.
    We have seen that there exists an index $ i \in [k] $ such that $ y^i $ is a vertex of $ Q_i $, $ w_i = 1 $ and $
    y^j = \zerovec $ for all $ j \in [k] \setminus \{i\} $.
    Since $ v = \pi(v') = \pi_i(y^i) \in P_i $ and the $ P_j $'s are pairwise disjoint, we obtain that $ i = \ell $.
    Finally, since $ (Q_i,\pi_i) $ is nicely $0/1$, $ y^i $ is uniquely determined by $ v $ and so is $ v' $.
\end{proof}
\noindent
Surprisingly, our computations considerably benefit from this simple consequence:
\begin{corollary}
    \label{cor:simplebalas}
    For any $ 0/1 $-polytope $ P \subseteq \R^p $ with $ \dim(P) \geq 1 $ and any point $ v \in \{0,1\}^p \setminus P $,
    it holds that
    \[
        \xcs\big(\conv(P \cup \{v\})\big) \leq \xcs(P) + 1. \qed
    \]
\end{corollary}
\noindent
Let us remark that the assumption of disjointness in the above statements is only required to satisfy
property~(\ref{enum:vertexbijective}) of Definition~\ref{def:nice01}.

\subsection{Reflections}
\label{sec:reflections}
Another general method to construct extended formulations for polytopes $ P $ has been introduced by Kaibel and
Pashkovich~\cite{KaibelP13}.
It is again based on the assumption that $ P $ can be written as the convex hull of the union of two polytopes $ P' $
and $ P'' $.
Further, it requires $ P'' $ to be very similar to $ P' $, namely being the reflection of $ P' $ through some
hyperplane.

For $ a \in \R^p \setminus \{ \zerovec \} $ and $ \beta \in \R $ let $ \Hleq(a,\beta) := \{ x : \langle a,x \rangle \leq
\beta \} $ denote the associated halfspace.
We further define $ \Heq(a,\beta) := \{ x : \langle a,x \rangle = \beta \} $ as the corresponding hyperplane.
The reflection $ \varphi_H \colon \R^p \to \R^p $ at $ H = \Heq(a,\beta) $ can be written as
\[
    \varphi_H(x) = x + 2 \cdotp \frac{\beta - \langle a,x \rangle}{\|a\|^2} \cdotp a.
\]
Suppose that $ P' \subseteq \Hleq(a,\beta) $ and $ P'' = \varphi_H(P') $.
Although \cite{KaibelP13} addresses to a slightly more general type of reflection, their results imply the simple
extended formulation
\[
    \conv(P' \cup P'') = \bigg\{ x + \lambda \cdotp \frac{2}{\|a\|^2} \cdotp a : x \in P', \, 0 \leq \lambda \leq \beta
    - \langle a,x \rangle \bigg\},
\]
where replacing $ P' $ again by an extension $ (Q',\pi') $ yields
\begin{equation}
    \label{eq:reflection}
    \conv(P' \cup P'') = \bigg\{ \pi'(y) + \lambda \cdotp \frac{2}{\|a\|^2} \cdotp a : y \in Q', \, 0 \leq \lambda \leq
    \beta - \langle a,\pi'(y) \rangle \bigg\}.
\end{equation}
In terms of a bound on the general extension complexity, we conclude:
\begin{theorem}[\cite{KaibelP13}]
    \label{thm:reflection}
    Let $ P' \subseteq \R^p $ and $ a \in \R^p \setminus \{ \zerovec \} $, $ \beta \in \R $ such that $ P' \subseteq
    \Hleq(a,\beta) $. Then
    \[
        \xc\Big(\conv\big(P' \cup \varphi_H(P')\big)\Big) \leq \xc(P') + 2.
    \]
\end{theorem}
\noindent
In Section~\ref{sec:computations}, we will make use of two very specific types of reflections:
First, for any $ i \in [p] $ let $ \flipmap \colon \R^p \to \R^p $ be the map that flips the $ i $th entry of vectors $
x \in \R^p $, i.e.,
\[
    \flipmap(x)_{\ell} := \begin{cases} 1-x_i & \text{ if } \ell = i \\
        x_{\ell} & \text{ else.} \end{cases}
\]
This map coincides with the reflection map $ \varphi_H $ that is induced by setting $ \beta = \pm 1 $, $ a_i = \pm2 $
and $ a_{\ell} = 0 $ for all $ \ell \in [p] \setminus \{i\} $.
Second, for any $ i,j \in [p] $ with $ i \neq j $ let $ \swapmap \colon \R^p \to \R^p $ be the map that swaps the $ i
$th and $ j $th coordinate of vectors $ x \in \R^p $, i.e.,
\[
    \varphi_{\mathrm{swap},i,j}(x)_{\ell} := \begin{cases} x_j & \text{ if } \ell = i \\
        x_i & \text{ if } \ell = j \\
        x_{\ell} & \text{ else.} \end{cases}
\]
This is equivalent to the reflection map that is induced by setting $ \beta = 0 $, $ a_i = \pm 1 $, $ a_j = \mp 1 $
and $ a_{\ell} = $ for all $ \ell \in [p] \setminus \{i,j\} $.

We say that $ (a,\beta) $ \emph{induces a symmetry of the cube} if it is of one of the above types.
Note that in all these cases, for all $ v \in \{0,1\}^p $, the expression $ \beta - \langle a,v \rangle $ attains only
values in $ \{ 0,t \} $ for $ t \in \{-1,1\} $ ($ t $ depends on the orientation of the associated halfspace we used).
This property allows us to make an analogous statement as in Theorem~\ref{thm:reflection} for sizes of nice $ 0/1
$-extensions:
\begin{theorem}
    \label{thm:reflectionsxcs}
    Let $ P' \subseteq \R^p $ be a $ 0/1 $-polytope and let $ a \in \R^p, \beta \in \R $ induce a symmetry of the cube
    with $ P' \subseteq \Hleq(a,\beta) $. Then
    \[
        \xcs\Big(\conv\big(P' \cup \varphi_H(P')\big)\Big) \leq \xcs(P') + 2.
    \]
\end{theorem}
\begin{proof}
    By the above paragraph, we may assume that $ \beta - \langle a,v \rangle \in \{0,1\} $ holds for all vertices
    $ v $ of $ P' $.
    Otherwise, replace $ P' $ by $ \varphi_H(P') $, which is affinely isomorphic and hence $ \xcs(P') =
    \xcs(\varphi_H(P')) $.

    Let $ (Q',\pi') $ be a nice $ 0/1 $-extension of $ P' $. As mentioned above, the polytope
    \[
        Q = \big\{ (y,\lambda) : y \in Q', \, 0 \leq \lambda \leq \beta - \langle a,\pi'(y) \rangle \big\}
    \]
    together with the map $ \pi(y,\lambda) = \pi'(y) + \lambda \cdotp \frac{2}{\|a\|^2} $ is an extension for $ P :=
    \conv\big(P' \cup \varphi_H(P')\big) $.
    We have to show that $ (Q,\pi) $ is nicely $ 0/1 $.

    Let $ (y,\lambda) $ be a vertex of $ Q $.
    First, suppose that there exists an $ \varepsilon > 0 $ such that $ 0 \leq \lambda - \varepsilon \leq \lambda \leq
    \lambda + \varepsilon \leq \beta - \langle a,\pi'(y) \rangle $.
    In this case, we can write
    \[
        (y,\lambda) = \frac{1}{2} (y,\lambda - \varepsilon) + \frac{1}{2} (y,\lambda + \varepsilon),
    \]
    a contradiction since $ (y,\lambda - \varepsilon), (y,\lambda + \varepsilon) \in Q $ and $ (y,\lambda) $ is assumed
    to be a vertex of $ Q $.
    Thus, at least one of the inequalities $ 0 \leq \lambda $ and $ \lambda \leq \beta - \langle a,\pi'(y) \rangle $
    holds with equality.

    Second, we show that $ y $ is a vertex of $ Q' $.
    Let us assume that there are vectors $ y^1,y^2 \in Q' $ with $ y^1 \neq y^2 $ such that $ y = \mu \cdotp y^1 +
    (1-\mu) \cdotp y^2 $ for some $ \mu \in (0,1) $.
    If $ \lambda = 0 $, we have that $ (y,\lambda) = \mu \cdotp (y^1,0) + (1-\mu) \cdotp (y^2,0) $ as well as $ (y^1,0),
    (y^2,0) \in Q $ since $ \pi'(Q') \subseteq \Hleq(a,\beta) $, a contradiction to $ (y,\lambda) $ being a vertex of $
    Q $.
    If otherwise $ \lambda = \beta - \langle a,\pi'(y) \rangle $, let us set $ \lambda_i := \beta - \langle a,\pi'(y^i)
    \rangle \geq 0 $ for $ i=1,2 $.
    By construction, we again have that $ (y^1,\lambda_1), (y^2,\lambda_2) \in Q $.
    Since
    \begin{align*}
        \mu \lambda_1 + (1-\mu) \lambda_2
        & = \mu (\beta - \langle a,\pi'(y^1) \rangle) + (1-\mu) (\beta - \langle a,\pi'(y^2) \rangle) \\
        & = \beta - \langle a,\pi'(y) \rangle = \lambda,
    \end{align*}
    we obtain $ (y,\lambda) = \mu \cdotp (y^1,\lambda_1) + (1-\mu) \cdotp (y^2,\lambda_1) $ and hence again a
    contradiction to the assumption that $ (y,\lambda) $ is a vertex of $ Q $.

    Since $ (Q',\pi') $ is nicely $ 0/1 $, $ y $ is a $ 0/1 $-vector and $ v := \pi'(y) $ is a vertex of $ P' $.
    We have seen that $ \lambda \in \{0, \beta - \langle a,v \rangle \} $ holds.
    By our first assumption, this implies that $ \lambda \in \{0,1\} $ and hence $ (y,\lambda) $ is a $ 0/1 $-vector.
    Further, we have that
    \[
        \pi(y,\lambda) \in \{ \pi(y,0), \pi(y,\beta - \langle a,v \rangle) \} = \{ v, \varphi_H(v) \}
    \]
    and since $ \varphi_H $ is one of the cube's symmetries, $ \varphi_H(v) $ is a $ 0/1 $ point and hence a vertex of $
    P $ and so is $ \pi(y,\lambda) $.

    It remains to be shown that every vertex $ v^* $ of $ P $ has a unique preimage.
    Since again $ \varphi_H $ is one of the cube's symmetries and $ P' \subseteq \Hleq(a,\beta) $, there exists a unique
    vertex $ v' $ of $ P' $ such that $ v^* = v' $ or $ v^* = \varphi_H(v') $.
    Thus, if $ (y^*,\lambda^*) $ projects onto $ v^* $, we must have $ \pi'(y^*) = v' $.
    Since $ (Q',\pi') $ is nicely $ 0/1 $, $ y^* $ is uniquely determined by $ v' $ (and hence by $ v^* $).
    Further, $ \lambda^* $ is also uniquely determined by $ v^* $ depending on whether $ v^* = v' $ or $ v^* =
    \varphi_H(v') $ holds.
    Note that if $ \varphi_H(v^*) = v^* $, then $ \lambda^* = 0 $ in both cases.
\end{proof}
\noindent
Another simple consequence which our computations exploit is the following fact:
\begin{corollary}
    \label{cor:reflectionfacet}
    Let $ P' \subseteq \R^p $ be a $ 0/1 $-polytope and let $ a \in \R^p, \beta \in \R $ induce a symmetry of the cube
    with $ P' \subseteq \Hleq(a,\beta) $. If $ P' \cap H $ is a facet of $ P' $, then
    \[
        \xcs\Big(\conv\big(P' \cup \varphi_H(P')\big)\Big) \leq \numfacets(P') + 1.
    \]
\end{corollary}
\begin{proof}
    Set $ k = \numfacets(P') $ and let $ A \in \R^{(k-1) \times p} $, $ b \in \R^{k-1} $ such that $ P' = \{ x \in \R^p
    : Ax \leq b, \, \langle a,x \rangle \leq \beta \} $.
    By (the proof of) Theorem~\ref{thm:reflectionsxcs}, we have that
    \[
        Q := \{ (x,\lambda) : x \in P', \, 0 \leq \lambda \leq \beta - \langle a,x \rangle \}.
    \]
    is a nice $ 0/1 $-extension (together with $ \pi(x,\lambda) = x + \lambda \cdotp a $) for $ \conv\big(P' \cup
    \varphi_H(P')\big) $.
    Since $ 0 \leq \lambda \leq \beta - \langle a,x \rangle $ implies $ \langle a,x \rangle \leq \beta $, we obtain
    \[
        Q = \{ (x,\lambda) : Ax \leq b, \, 0 \leq \lambda \leq \beta - \langle a,x \rangle \}
    \]
    and hence $ Q $ has at most $ (k - 1) + 2 = k + 1 $ facets.
\end{proof}
\subsection{Down-Monotonicity}
The third construction that we will use in Section~\ref{sec:computations} is related to the concept of down-monotone
polyhedra.
Here, we are only interested in down-monotonicity with respect to only one coordinate.
Let us consider the map $ \down_j \colon \R^p \to \R^p $ defined via
\[
    \down_j(x)_i := \begin{cases}
        0 & \text{if } j = k \\
        x_i & \text{else},
    \end{cases}
\]
where $ j \in [p] $.
For a polyhedron $ P' \subseteq \R^p_+ $, it is straightforward to see that
\[
    \conv \big(P' \cup \down_j(P')\big) =
    \big\{ z \in \R^p : x \in P', \, 0 \leq z_j \leq x_j, \, z_i = x_i \ \forall \, i \in [p] \setminus \{ j \} \big\}
\]
holds, which immediatly yields the bound
\[
    \xc \Big( \conv \big(P' \cup \down_j(P') \big) \Big) \leq \xc(P') + 2.
\]
In terms of nice $ 0/1 $-extensions we need an additional requirement on $ P' $:
\begin{theorem}
    \label{thm:downward}
    Let $ P' $ be a $ 0/1 $-polytope such that $ \down_j $ is injective on the vertices of $ P' $. Then
    \[
        \xcs \Big( \conv \big(P' \cup \down_j(P') \big) \Big) \leq \xcs(P') + 2.
    \]
\end{theorem}
\begin{proof}
    Let $ (Q',\pi') $ be a nice $ 0/1 $-extension of $ P' $ and consider the polytope
    \[
        Q := \{ (y,\lambda) : y \in Q', \, 0 \leq \lambda \leq \pi'(y)_j \}
    \]
    together with the linear map
    \[
        \pi(y,\lambda) := \down_j(\pi'(y)) + \unitvec_j \cdotp \lambda,
    \]
    where $ \unitvec_j $ is the $ j $th unit vector.
    As mentioned above, $ (Q,\pi) $ is an extension for $ P := \conv (P' \cup \down_j(P') ) $.
    Let $ (y,\lambda) $ be a vertex of $ Q $.
    Analogous to the proof of Theorem~\ref{thm:reflectionsxcs}, it is easy to see that $ y $ is a vertex of $ Q' $ and $
    \lambda \in \{0,\pi'(y)_j\} \subseteq \{0,1\} $.
    This directly implies that $ Q $ is a $ 0/1 $-polytope and that every vertex of $ Q $ is projected onto a vertex of
    $ P $.
    Finally, the uniqueness of preimages of vertices of $ P $ follows from the fact that $ \down_j $ is injective on the
    vertices of $ P' $ and that $ (Q',\pi') $ be a nice $ 0/1 $-extension of $ P' $.
\end{proof}
 
\section{Combinatorial Lower Bounds}
\label{sec:lowerbounds}

\noindent
In this section, we review known combinatorial lower bounds on the extension complexities of polytopes.
We present a slight refinement of the rectangle covering bound, for which, in the case of $ 4 $-dimensional $ 0/1
$-polytopes, the resulting bound (i) is still easy to compute and (ii) provides tight results.
Before, let us consider the following very simple bounds on the extension complexity:
\begin{proposition}
    \label{prop:triviallowerbounds}
    For any polytope $ P $ with $ \dim(P) = d $ it holds that
    \begin{itemize}
        \item[a)] $ \xc(P) \geq d + 1 $,
        \item[b)] $ \xc(P) = d + 1 $ if and only if $ P $ is a simplex,
        \item[c)] $ \xc(P) = d + 2 $ if and only if $ \min\{\numfacets(P),\numverts(P)\} = d + 2 $.
    \end{itemize}
\end{proposition}
\begin{proof}
    Let $ (Q,\pi) $ be a minimum size extension for $ P $.
    By Lemma~\ref{lem:polytope}, we may assume that $ Q $ is a polytope.
    Parts a) and b) are left to the reader.
    By Proposition~\ref{prop:trivialupperbounds} it remains to show the only-if part of c).

    Suppose that $ \numfacets(P),\numverts(P) \geq d + 3 $ and let us assume that $ Q $ has at most $ d + 2 $ facets.
    Note that $ \dim(Q) \geq d + 1 $ since otherwise $ Q $ is isomorphic to $ P $ and hence has at least $ d + 3 $
    facets.
    This implies that $ Q $ is a $ (d+1) $-simplex and hence has only $ d+2 $ vertices.
    Since $ Q $ must have at least as many vertices as $ P $, we obtain a contradiction.
\end{proof}
\noindent
It turns out that, together with constructions from Section~\ref{sec:upperbounds}, the above bounds suffice to determine
the extension complexities of all $ 0/1 $-polytopes up to dimension $ 3 $ (see Section~\ref{sec:computations}).
Not surprisingly, one needs more profound bounds for tight results in dimension $ 4 $.

\subsection{Yannakakis' Theorem}
In his seminal paper~\cite{Yannakakis91}, Yannakakis gave an algebraic interpretation of the extension complexity and
laid the foundation of the developtment of bounds used in important theoretical results as in~\cite{FioriniMPTW12}
or~\cite{Rothvoss13}.

Let $ P = \{ x \in \R^p : Ax \leq b \} $ be a polytope with $ A \in \R^{m \times p} $, $ b \in \R^m $ and $
\{v^1,\dotsc,v^k\} $ the set of its vertices.
The nonnegative matrix $ S \in \R_+^{m \times k} $ defined via
\[
    S_{i,j} := b_i - \langle A_{i,*} , v^j \rangle,
\]
where $ A_{i,*} $ denotes the $ i $th row of $ A $, is called a \emph{slack matrix} of $ P $.
Further, the smallest number $ r $ such that $ S = U \cdotp V $ for two nonnegative matrices $ U \in \R_+^{m \times r}
$, $ V \in \R_+^{r \times k} $ is called the \emph{nonnegative rank} of $ S $ and denoted by $ r_+(S) $.
\begin{theorem}[Yannakakis '91]
    Let $ P $ be a polytope and $ S $ be a slack matrix of $ P $. Then $ \xc(P) = r_+(S) $.
\end{theorem}
\noindent
The nonnegative rank of a matrix $ S $ can also be seen as the smallest $ r $ such that $ S $ can be written as the sum
of $ r $ nonnegative rank-1 matrices, which is a helpful interpretation to obtain combinatorial bounds.
\subsection{Rectangle Coverings and Fooling Sets}
A well-known bound on the nonnegative rank is the \emph{rectangle covering bound}, which we will present here.
For more details, in particular related to its application in the field of extended formulations, we refer to the paper
of Fiorini et al.~\cite{FioriniKPT11}.

Let $ S $ be a slack matrix whose rows and columns are indexed by some sets $ \mathcal{I} $ and $ \mathcal{J} $,
respectively.
Let us define the \emph{support} of $ S $ as the set $ \supp(S) := \{ (i,j) \in \mathcal{I} \times \mathcal{J} :
S_{i,j} > 0 \} $.
A set $ I \times J $ with $ I \subseteq \mathcal{I} $, $ J \subseteq \mathcal{J} $ is now called a \emph{rectangle}, if
$ I \times J \subseteq \supp(S) $.
Further, a set of rectangles $ R_1,\dotsc,R_k $ is called a \emph{rectangle covering} if $ \supp(S) = \cup_{\ell=1}^k
R_{\ell} $.
The following observation motivates to consider the so-called \emph{rectangle covering number} of $ S $, which is
denoted by $ \rc(S) $ and defined as the smallest number of rectangles in any rectangle covering of $ S $.
Suppose that the nonnegative rank of $ S $ is $ r $, i.e., there exist nonnegative matrices $
\mathcal{R}^1,\dotsc,\mathcal{R}^r $ such that $ S = \sum_{\ell=1}^r \mathcal{R}^{\ell} $.
Then the sets $ R_{\ell} := \supp(\mathcal{R}_\ell) $ are clearly rectangles.
Moreover, one has that
\[
    \supp(S) = \supp(\cup_{\ell=1}^r \mathcal{R}_{\ell}) = \cup_{\ell=1}^r \supp(\mathcal{R}_\ell)
    = \cup_{\ell=1}^r R_{\ell},
\]
and hence the $ R_{\ell} $'s form a rectangle covering of $ S $.
Thus, if $ P $ is a polytope and $ S $ a slack matrix of $ P $, we obtain the rectangle covering bound
\[
    \rc(S) \leq r_+(S) = \xc(P).
\]
Since it still seems to be a difficult task to determine (or compute) $ \rc(S) $, one is of course interested in further
lower bounds that are easier to compute.
One simple bound on the rectangle covering number is the \emph{fooling set bound}, where a \emph{fooling set} is a set $
F \subseteq \supp(S) $ such that
\[
    S_{i_1,j_2} = 0 \text{ or } S_{i_2,j_1} = 0
\]
holds for all distinct pairs $ (i_1,j_1), (i_2,j_2) \in F $.
In what follows, the largest cardinality of a fooling set of $ S $ is called the \emph{fooling set number} and will be
denoted by $ \omega(S) $.
It is easy to see that any rectangle of $ S $ can contain at most one element of $ F $.
Thus, any rectangle covering of $ S $ consists of at least $ |F| $ rectangles and hence we obtain
\[
    \omega(S) \leq \rc(S).
\]
In summary:
\begin{proposition}
    Let $ P $ be a polytope and $ S $ a slack matrix of $ P $. Then
    \[
        \omega(S) \leq \rc(S) \leq r_+(S) = \xc(P).
    \]
\end{proposition}
\noindent
While it is not much known about the general performance of the rectangle covering bound, there are some definite
limitations.
For instance, it is a classical fact that $ \omega(S) \leq (\dim(P) + 1)^2 $, see, e.g., \cite{FioriniKPT11}.
Of course, in the case of $ 0/1 $-polytopes of dimension $ 4 $, this limitation is trivial since such polytopes have at
most $ 16 $ vertices and hence an extension complexity of at most $ 16 $.
In fact, it turns out that the fooling set bound already yields tight bounds in many of our computations.
\subsection{Refinined Rectangle Coverings}
Although the classical lower bounds perform surprisingly well on slack matrices of $ 0/1 $-polytopes of dimension $ 4 $,
there are still some polytopes for which there is a gap between the rectangle covering number and the extension
complexity.
One drawback of the rectangle covering number is that it only depends on the sparsity pattern of the considered matrix.
In general, a rectangle covering $ R_1,\dotsc,R_k $ of a nonnegative matrix $ S $ might be far away from being induced
by nonnegative rank-1 matrices $ \mathcal{R}_1,\dotsc,\mathcal{R}_k $ such that $ S = \sum_{\ell=1}^k \mathcal{R}_{\ell}
$.
For instance, the matrix
\[
    \begin{pmatrix} 2 & 1 \\ 1 & 1 \end{pmatrix}
\]
has rectangle covering number $ 1 $ while its nonnegative rank is $ 2 $.
Thus, in order to assure that a rectangle covering is at least locally (by only considering $ 2 \times 2 $-submatrices)
extendable to a rank-1 decomposition, we propose the following additional requirement:
\begin{definition}
    \label{def:refinedreccover}
    A rectangle covering $ R_1,\dotsc,R_k $ of a nonnegative matrix $ S $ is called a \emph{refined covering} if
    \[
        \Big| \Big\{ \ell : R_\ell \cap \{ (i_1,j_1), (i_2,j_2) \} \neq \emptyset \Big\} \Big| \geq 2
    \]
    holds for all pairs $ (i_1,j_1), (i_2,j_2) $ with $ S_{i_1,j_1} \cdot S_{i_2,j_2} > S_{i_1,j_2} \cdot S_{i_1,j_2} $.
    The smallest size of any refined covering of $ S $ is called the \emph{refined rectangle covering number} and
    denoted by $ \rrc(S) $.
\end{definition}
\noindent
It turns out that this quantity allows us to close all remaining gaps in our computations in
Section~\ref{sec:computations}.
\begin{theorem}
    Let $ S $ be a nonnegative matrix. Then
    \[
        \rc(S) \leq \rrc(S) \leq r_+(S).
    \]
\end{theorem}
\begin{proof}
    Suppose that there exist nonnegative rank-1 matrices $ \mathcal{R}_1,\dotsc,\mathcal{R}_k $ such that
    \begin{equation}
        \label{eq:proofrefinedsum}
        S = \sum_{i=1}^k \mathcal{R}_i.
    \end{equation}
    We have already seen that $ \supp(\mathcal{R}_1), \dotsc, \supp(\mathcal{R})_k $ is a rectangle covering of $ S $.
    It suffices to show that this covering also satisfies the requirements of Definition~\ref{def:refinedreccover}.
    Let us assume the contrary.
    By reordering the rows and columns of $ S $, we may assume that
    \begin{equation}
        \label{eq:proofrefineddeterminant}
        S_{1,1} \cdotp S_{2,2} > S_{1,2} \cdotp S_{2,1}
    \end{equation}
    and let $ j \in [k] $ such that $ \mathcal{R} = \mathcal{R}_j $ is the only matrix among the $ \mathcal{R}_i $'s
    whose support has a non-empty intersection with $ \{ (1,1), (2,2) \} $.
    By equation~\eqref{eq:proofrefinedsum}, this implies that
    \begin{equation}
        \label{eq:proofrefinedequal}
        \mathcal{R}_{1,1} = S_{1,1} \quad \text{and} \quad \mathcal{R}_{2,2} = S_{2,2}.
    \end{equation}
    Since $ \mathcal{R} $ has rank 1, we further have that
    \begin{equation}
        \label{eq:proofrefinedrank1}
        \mathcal{R}_{1,1} \cdotp \mathcal{R}_{2,2} = \mathcal{R}_{1,2} \cdotp \mathcal{R}_{2,1}.
    \end{equation}
    By the nonnegativity of all $ \mathcal{R}_i $'s, we finally obtain
    \[
        S_{1,2} \cdotp S_{2,1} \geq \mathcal{R}_{1,2} \cdotp \mathcal{R}_{2,1}
        \stackrel{\eqref{eq:proofrefinedrank1}}{=} \mathcal{R}_{1,1} \cdotp \mathcal{R}_{2,2}
        \stackrel{\eqref{eq:proofrefinedequal}}{=} S_{1,1} \cdotp S_{2,2}
        \stackrel{\eqref{eq:proofrefineddeterminant}}{>} S_{1,2} \cdotp S_{2,1},
    \]
    a contradiction.
\end{proof}
 
\section{Computation \& Results}
\label{sec:computations}
\noindent
In this section, we briefly describe our approach of computing the extension complexities of all $ 0/1 $-polytopes up to
dimension $ 4 $ and present the results.
In order to provide comprehensible results, we assign an ID to each polytope.
For this purpose, let us define the function $ b \colon \{0,1\}^n \to \{0,\dotsc,2^n-1\} $ via $ b(v) := \sum_{i=1}^n
v_i 2^{i-1} $.
For a set $ V \subseteq \{0,1\}^n $ of binary vectors, the ID of the corresponding polytope $ P = \conv(V) $ is now
defined as
\[
    \mathrm{ID}(P) := \sum_{v \in V} b(v) 2^{b(v)} \in \{0,2^{2^n}-1\}.
\]
Since the extension complexity of a polytope (as well as our notion of $ \xcs(\cdot) $) is invariant under affine
transformations, we shall divide all $ 0/1 $-polytopes in affine equivalence classes.
The representative of each equivalence class is chosen to be the polytope with smallest ID inside the class.

In order to compute the affine equivalence classes, we first enumerated all $ 0/1 $-equivalence classes as proposed
in~\cite{Aichholzer00}.
Recall that two $ 0/1 $-polytopes $ P, P' \subseteq \R^n $ are $ 0/1 $-equivalent if there exists an affine isomorphism
$ f \colon \{0,1\}^n \to \{0,1\}^n $ with $ f(\{0,1\}^n) = \{0,1\}^n $ such that $ f(P) = P' $.
After computing the f-vectors of the representatives of each $ 0/1 $-equivalence class via \polymake, it remained to run a
small number of tests for affine equivalence.
See Table~\ref{tab:eqclasses} for the intermediate results.

\begin{table}
    \small
    \begin{tabular}{crrr}
        \toprule
        vertices & polytopes & $0/1$-equivalence classes & affine equiv. classes \\
        \midrule
        5 & 3008 & 17 & 1 \\
6 & 7408 & 40 & 8 \\
7 & 11280 & 54 & 17 \\
8 & 12850 & 72 & 36 \\
9 & 11440 & 56 & 40 \\
10 & 8008 & 50 & 43 \\
11 & 4368 & 27 & 26 \\
12 & 1820 & 19 & 19 \\
13 & 560 & 6 & 6 \\
14 & 120 & 4 & 4 \\
15 & 16 & 1 & 1 \\
16 & 1 & 1 & 1 \\
\midrule
$\sum$ & 60879 & 347 & 202 \\ 
        \bottomrule
    \end{tabular}
    \caption{Number of $ 4 $-dimensional $ 0/1 $-polytopes}
    \label{tab:eqclasses}
\end{table}

Before we restrict ourselves to $ 4 $-dimensional $ 0/1 $-polytopes, let us show that no further (computer aided)
computations are needed to determine the extension complexities of $ 0/1 $-polytopes of dimension up to $ 3 $.

\subsection{Computations Up to Dimension 3}
With respect to affine equivalence, there are $ 12 $ different $ 0/1 $-polytopes of dimension $ 3 $ or less, see
Figure~\ref{fig:3dim}.
Applying Proposition~\ref{prop:triviallowerbounds} already yields that in each of the cases, except for the \emph{sliced
cube} (ID $ 127 $), the trivial extensions of Proposition~\ref{prop:trivialupperbounds} are smallest possible.
Considering the sliced cube, observe that it can be written as the convex hull of the union of the prism and a single
point.
Since the prism has $ 5 $ facets, we obtain an extension with only $ 6 $ facets by Corollary~\ref{cor:simplebalas}.
Again by Proposition~\ref{prop:triviallowerbounds}, this has smallest possible size.
Note that all minimum size extensions used here are nice $ 0/1 $-extensions.
A summary of the results can be found in Table~\ref{tab:dim3}.

\begin{figure}
    \begin{center}
\begin{small}
\begin{tikzpicture}[scale=1]

\tikzstyle{thickline}=[line width=1.5pt, line join=round]
\tikzstyle{thinline}=[]
\tikzstyle{dottedline}=[dotted]

\newcommand{\dx}{0.4}
\newcommand{\dy}{0.3}

\def\drawname#1#2#3{
    \node at (#1+0.7,#2-0.25) {#3};
}

\def\drawdottedcube#1#2{
    \draw[dottedline] (#1,#2) -- ++(1,0) -- ++ (0,1) -- ++(-1,0) -- cycle;
    \draw[dottedline] (#1+\dx,#2+\dy) -- ++(1,0) -- ++ (0,1) -- ++(-1,0) -- cycle;
    \draw[dottedline] (#1,#2) -- ++(\dx,\dy);
    \draw[dottedline] (#1+1,#2) -- ++(\dx,\dy);
    \draw[dottedline] (#1,#2+1) -- ++(\dx,\dy);
    \draw[dottedline] (#1+1,#2+1) -- ++(\dx,\dy);
}

\def\drawpoint#1#2{
    \drawdottedcube{#1}{#2}
    \node[fill,circle,inner sep=0pt,minimum size=3pt] at (#1,#2) {};
    \draw[thickline] (#1,#2) -- (#1,#2);
    \drawname{#1}{#2}{point ($1$)}
}
\def\drawinterval#1#2{
    \drawdottedcube{#1}{#2}
    \draw[thickline] (#1,#2) -- +(1,0);
    \drawname{#1}{#2}{interval ($3$)}
}
\def\drawtriangle#1#2{
    \drawdottedcube{#1}{#2}
    \draw[thickline] (#1,#2) -- +(1,0) -- +(\dx,\dy) -- cycle;
    \drawname{#1}{#2}{triangle ($7$)}
}
\def\drawsquare#1#2{
    \drawdottedcube{#1}{#2}
    \draw[thickline] (#1,#2) -- ++(1,0) -- ++(\dx,\dy) -- ++(-1,0) -- cycle;
    \drawname{#1}{#2}{square ($15$)}
}

\def\drawsimplex#1#2{
    \drawdottedcube{#1}{#2}
    \draw[thinline] (#1,#2) -- +(\dx,\dy) -- +(1,0);
    \draw[thinline] (#1,#2+1) -- +(\dx,\dy-1);
    \draw[thickline] (#1,#2) -- +(1,0) -- +(0,1) -- cycle;
    \drawname{#1}{#2}{tetrahedron ($23$)}
}

\def\drawpyramid#1#2{
    \drawdottedcube{#1}{#2}
    \draw[thickline] (#1,#2) -- +(1,0) -- +(\dx+1,\dy) -- +(0,1) -- cycle;
    \draw[thickline] (#1+1,#2) -- +(-1,1);
    \draw[thinline] (#1,#2) -- ++(\dx,\dy) -- ++(1,0);
    \draw[thinline] (#1,#2+1) -- ++(\dx,\dy-1);
    \drawname{#1}{#2}{pyramid ($31$)}
}

\def\drawbipyramid#1#2{
    \drawdottedcube{#1}{#2}
    \draw[thinline] (#1,#2) -- ++(\dx+1,\dy) -- ++(-1,1);
    \draw[thickline] (#1,#2) -- +(1,0) -- +(\dx+1,\dy) -- +(1,1) -- +(\dx,\dy+1) -- cycle;
    \draw[thickline] (#1,#2) -- ++(1,1) -- ++(0,-1);
    \drawname{#1}{#2}{bipyramid ($107$)}
}

\def\drawprism#1#2{
    \drawdottedcube{#1}{#2}
    \draw[thinline] (#1,#2) -- ++(\dx,\dy) -- ++(1,0);
    \draw[thinline] (#1,#2+1) -- ++(\dx,\dy-1);
    \draw[thickline] (#1,#2) -- ++(1,0) -- ++(\dx,\dy) -- ++(-1*\dx,-1*\dy+1) -- ++(-1,0) -- cycle;
    \draw[thickline] (#1+1,#2) -- ++(0,1);
    \drawname{#1}{#2}{prism ($63$)}
}

\def\drawnameless#1#2{
    \drawdottedcube{#1}{#2}
    \draw[thinline] (#1,#2) -- ++(\dx,\dy) -- ++(1,0) -- ++(-1,1) -- ++(0,-1);
    \draw[thickline] (#1,#2) -- ++(1,0) -- ++(0,1) -- ++(-1,-1) -- ++(\dx,\dy+1) -- ++(-1*\dx+1,-1*\dy) -- ++(\dx,\dy-1)
    -- ++(-1*\dx,-1*\dy);
    \drawname{#1}{#2}{nameless ($111$)}
}

\def\drawoctahedron#1#2{
    \drawdottedcube{#1}{#2}
    \draw[thinline] (#1+\dx,#2+\dy) -- ++(1,0) -- ++(-1,1) -- cycle;
    \draw[thickline] (#1+1,#2) -- ++(0,1) -- ++(-1,0) -- ++(1,-1) -- ++(\dx,\dy) -- ++(-1*\dx,-1*\dy+1) -- ++(\dx-1,\dy)
    -- ++(-1*\dx,-1*\dy) -- ++(\dx,\dy-1) -- cycle;
    \drawname{#1}{#2}{octahedron ($126$)}
}

\def\drawslicedcube#1#2{
    \drawdottedcube{#1}{#2}
    \draw[thinline] (#1,#2) -- ++(\dx,\dy) -- ++(1,0) -- ++(-1,1) -- ++(0,-1);
    \draw[thickline] (#1,#2) -- ++(1,0) -- ++ (0,1) -- ++(-1,0) -- cycle;
    \draw[thickline] (#1+1,#2) -- +(\dx,\dy) -- +(0,1) -- +(\dx-1,\dy+1) -- +(-1,1);
    \drawname{#1}{#2}{sliced cube ($127$)}
}

\def\drawcube#1#2{
    \draw[thinline] (#1,#2) -- ++(\dx,\dy) -- ++(1,0);
    \draw[thinline] (#1+\dx,#2+\dy) -- ++(0,1);
    \draw[thickline] (#1,#2+1) -- ++(\dx,\dy) -- ++(1,0) -- ++(0,-1) -- ++(-1*\dx,-1*\dy) -- ++(-1,0) -- ++(0,1) --
    ++(1,0) -- ++(0,-1);
    \draw[thickline] (#1+1,#2+1) -- ++(\dx,\dy);
    \drawname{#1}{#2}{cube ($255$)}
}

\drawpoint{-2.5}{7.5}
\drawinterval{0}{7.5}
\drawtriangle{2.5}{7.5}
\drawsquare{5}{7.5}

\drawsimplex{-2.5}{5}
\drawpyramid{0}{5}
\drawbipyramid{2.5}{5}
\drawprism{5}{5}

\drawnameless{-2.5}{2.5}
\drawoctahedron{0}{2.5}
\drawslicedcube{2.5}{2.5}
\drawcube{5}{2.5}

\end{tikzpicture}
\end{small}
\end{center}
 
    \caption{All $ 0/1 $-polytopes up to dimension $ 3 $ (IDs in parentheses)}
    \label{fig:3dim}
\end{figure}
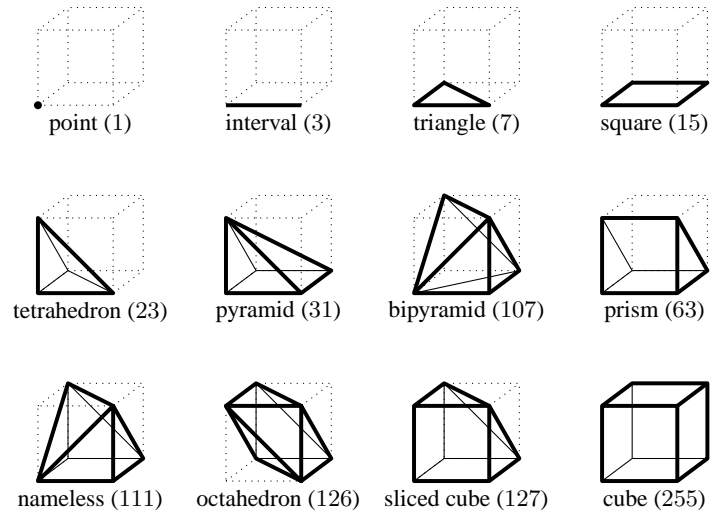

\begin{table}
    \small
    \begin{tabular}{lrcccc}
        \toprule
        polytope & \multicolumn{1}{c}{ID} & dimension & vertices & facets & $ \xcs $ \\
        \midrule
        point & 1 & 0 & 1 & 0 & 0 \\
        interval & 3 & 1 & 2 & 2 & 2 \\
        triangle & 7 & 2 & 3 & 3 & 3 \\
        square & 15 & 2 & 4 & 4 & 4 \\
        tetrahedron & 23 & 3 & 4 & 4 & 4 \\
        pyramid & 31 & 3 & 5 & 5 & 5 \\
        bipyramid & 107 & 3 & 5 & 6 & 5 \\
        prism & 63 & 3 & 6 & 5 & 5 \\
        nameless & 111 & 3 & 6 & 7 & 6 \\
        octahedron & 126 & 3 & 6 & 8 & 6 \\
        sliced cube & 127 & 3 & 7 & 7 & 6 \\
        cube & 255 & 3 & 8 & 6 & 6 \\
        \bottomrule
    \end{tabular}
    \caption{Extension complexities of representatives of all affine equivalence classes of $ 0/1 $-polytopes up to
    dimension $ 3 $}
    \label{tab:dim3}
\end{table}

\subsection{Computations in Dimension 4}
In Section~\ref{sec:lowerbounds}, we presented several lower bounds on the extension complexity.
Let us recall, that we have the inequality chain
\[
    \omega(S) \leq \rc(S) \leq \rrc(S) \leq r_+(S) = \xc(P) \leq \xcs(P),
\]
where $ S $ is the slack matrix of the polytope $ P $.
For each representative of all affine equivalence classes, we computed their slack matrices with \polymake{} and
calculated all required lower bounds by using simple, exact backtracking algorithms.

In order to obtain tight upper bounds on $ \xc(P) $, we first fixed all representatives for which the trivial extension
of Proposition~\ref{prop:trivialupperbounds} is already of minimum size.
Observe now that all upper bounds in Section~\ref{sec:upperbounds} are induced by extensions constructed in the
following way:
Start with a polytope $ P' $ and perform a simple geometric operation to obtain $ P $.
If the size of the resulting extension matches the largest lower bound on $ \xc(P) $, then the extension complexity of $
P $ is determined.
In that case, we call the polytope $ P' $ the \emph{predecessor polytope} of $ P $.
By iterating this process of checking for appropriate predecessors and geometric operations that yield smallest possible
extensions, we were able to determine the extension complexities of all representatives.
Note that these computations can also be performed by very simple algorithms.

The final results are presented in Table~\ref{tab:finalresults} and can be read as follows:
For each representative $ P $, we list its ID, its number of vertices ($ n $) and facets ($ m $) as well as all computed
lower bounds on $ \xc(P) $ as well as $ \xc(P) $ itself.
Since all implicitly computed extensions are even nice $ 0/1 $-extensions, we have that $ \xc(P) $ and $ \xcs(P) $
conincide in all these cases.
Further, we indicate the geometric operation used to construct the smallest extension as well as the corresponding
predecessor polytope.
For the geometric operations, we use the following symbols:
\[
\begin{array}{cl}
    \text{-} & \text{no operation, original outer description is smallest possible} \\
    \Delta & \text{trivial vertex-extension (see Proposition~\ref{prop:trivialupperbounds})} \\
    \cup & \text{union with a single point (see Corollary~\ref{cor:simplebalas})} \\
    \div & \text{reflection at a hyperplane corresponding to one of the cube's symmetries} \\
    & \text{(see Theorem~\ref{thm:reflectionsxcs})} \\
    \phantom{{}^*} \div^* & \text{reflection at a hyperplane corresponding to one of the cube's symmetries that is} \\
    & \text{also a facet of the predecessor polytope (see Corollary~\ref{cor:reflectionfacet})} \\
    \downarrow & \text{making the predecessor down-monotone with respect to one coordinate} \\
    & \text{(see \ref{thm:downward})}
\end{array}
\]
Finally, since it may be too time-consuming for the reader to compute the affine equivalence class of a given
polytope $ P $, we additionaly provide the representatives of all $ 0/1 $-equivalence classes that fall into the same
affine equivalence class as $ P $.
Thus, the reader has only to compute (the representative of) the $ 0/1 $-equivalence class of $ P $.
Note that this can be done very efficiently since any affine map of the cube is a composition of some entry flips
and coordinate swaps defined in Section~\ref{sec:reflections}, see~\cite{Aichholzer00}.

\vspace{1em}

{
    \small
    \tablefirsthead{        \toprule
        \multicolumn{1}{c}{ID} & $ n $ & $ m $ & $ \omega $ & $ \rc $ & $ \rrc $ & $ \xcs $ & extension & predecessor &
        further representatives \\
        \midrule
    }
    \tablehead{        \multicolumn{10}{l}{\textit{continued from previous page}} \\
        \midrule
        \multicolumn{1}{c}{ID} & $ n $ & $ m $ & $ \omega $ & $ \rc $ & $ \rrc $ & $ \xcs $ & extension & predecessor &
        further representatives \\
        \midrule
    }
    \tabletail{        \midrule
        \multicolumn{10}{r}{\textit{continued on next page}} \\
    }
    \tablelasttail{        \bottomrule
    }
    \bottomcaption{Extension complexities of representatives of all affine equivalence classes of $ 4 $-dimensional $
    0/1 $-polytopes}
    \label{tab:finalresults}
    \begin{supertabular}{rrrrrrrcp{2.5cm}p{3cm}}
        279 & 5 & 5 & 5 & 5 & 5 & 5 & - & - & \fontsize{0.2cm}{1em}{}\selectfont{}283, 286, 301, 361, 362, 391, 395, 406, 410, 425, 428, 488, 856, 872, 1681, 5761 \\
287 & 6 & 6 & 6 & 6 & 6 & 6 & - & - & \fontsize{0.2cm}{1em}{}\selectfont{}303, 317, 318, 366, 399, 411, 427, 429, 430, 444, 490, 858, 876, 965, 966, 980, 984, 1635, 1641, 1650, 1656, 1686, 5766 \\
363 & 6 & 7 & 6 & 6 & 6 & 6 & $ \Delta $ & - & \fontsize{0.2cm}{1em}{}\selectfont{}407, 414, 489, 1713, 1714, 1716 \\
5769 & 6 & 8 & 6 & 6 & 6 & 6 & $ \Delta $ & - & \fontsize{0.2cm}{1em}{}\selectfont{}- \\
5784 & 6 & 8 & 6 & 6 & 6 & 6 & $ \Delta $ & - & \fontsize{0.2cm}{1em}{}\selectfont{}- \\
6017 & 6 & 8 & 6 & 6 & 6 & 6 & $ \Delta $ & - & \fontsize{0.2cm}{1em}{}\selectfont{}- \\
854 & 6 & 9 & 6 & 6 & 6 & 6 & $ \Delta $ & - & \fontsize{0.2cm}{1em}{}\selectfont{}857, 874 \\
873 & 6 & 9 & 6 & 6 & 6 & 6 & $ \Delta $ & - & \fontsize{0.2cm}{1em}{}\selectfont{}1683 \\
5763 & 6 & 9 & 6 & 6 & 6 & 6 & $ \Delta $ & - & \fontsize{0.2cm}{1em}{}\selectfont{}- \\
319 & 7 & 6 & 6 & 6 & 6 & 6 & - & - & \fontsize{0.2cm}{1em}{}\selectfont{}431, 494, 829, 892, 967, 988, 1639, 1654, 1912, 5782 \\
1643 & 7 & 7 & 7 & 7 & 7 & 7 & - & - & \fontsize{0.2cm}{1em}{}\selectfont{}1718 \\
367 & 7 & 8 & 7 & 7 & 7 & 7 & $ \Delta $ & - & \fontsize{0.2cm}{1em}{}\selectfont{}415, 446, 491, 1777, 1778, 1969, 1972 \\
855 & 7 & 8 & 7 & 7 & 7 & 7 & $ \Delta $ & - & \fontsize{0.2cm}{1em}{}\selectfont{}859, 862, 878, 981, 985, 1651, 1658 \\
382 & 7 & 9 & 7 & 7 & 7 & 7 & $ \Delta $ & - & \fontsize{0.2cm}{1em}{}\selectfont{}445, 2017, 2018, 5737, 5738 \\
875 & 7 & 9 & 7 & 7 & 7 & 7 & $ \Delta $ & - & \fontsize{0.2cm}{1em}{}\selectfont{}877, 982, 1715, 1717 \\
1657 & 7 & 9 & 7 & 7 & 7 & 7 & $ \Delta $ & - & \fontsize{0.2cm}{1em}{}\selectfont{}1687, 1721 \\
5774 & 7 & 10 & 7 & 7 & 7 & 7 & $ \Delta $ & - & \fontsize{0.2cm}{1em}{}\selectfont{}- \\
5785 & 7 & 10 & 7 & 7 & 7 & 7 & $ \Delta $ & - & \fontsize{0.2cm}{1em}{}\selectfont{}- \\
5786 & 7 & 10 & 7 & 7 & 7 & 7 & $ \Delta $ & - & \fontsize{0.2cm}{1em}{}\selectfont{}- \\
6019 & 7 & 10 & 7 & 7 & 7 & 7 & $ \Delta $ & - & \fontsize{0.2cm}{1em}{}\selectfont{}- \\
6025 & 7 & 10 & 7 & 7 & 7 & 7 & $ \Delta $ & - & \fontsize{0.2cm}{1em}{}\selectfont{}- \\
5767 & 7 & 11 & 7 & 7 & 7 & 7 & $ \Delta $ & - & \fontsize{0.2cm}{1em}{}\selectfont{}- \\
5771 & 7 & 11 & 7 & 7 & 7 & 7 & $ \Delta $ & - & \fontsize{0.2cm}{1em}{}\selectfont{}- \\
5801 & 7 & 12 & 7 & 7 & 7 & 7 & $ \Delta $ & - & \fontsize{0.2cm}{1em}{}\selectfont{}- \\
5804 & 7 & 12 & 7 & 7 & 7 & 7 & $ \Delta $ & - & \fontsize{0.2cm}{1em}{}\selectfont{}6040 \\
6625 & 7 & 13 & 7 & 7 & 7 & 7 & $ \Delta $ & - & \fontsize{0.2cm}{1em}{}\selectfont{}- \\
831 & 8 & 6 & 6 & 6 & 6 & 6 & - & - & \fontsize{0.2cm}{1em}{}\selectfont{}975, 1020, 15555 \\
863 & 8 & 7 & 7 & 7 & 7 & 7 & - & - & \fontsize{0.2cm}{1em}{}\selectfont{}989, 1655, 1910, 1914 \\
383 & 8 & 8 & 7 & 7 & 7 & 7 & $ \cup $ & $ 375 \ (\cong 319) $ & \fontsize{0.2cm}{1em}{}\selectfont{}447, 495, 510, 2033, 2034, 2040 \\
893 & 8 & 8 & 7 & 7 & 7 & 7 & $ \cup $ & $ 892 \ (\cong 319) $ & \fontsize{0.2cm}{1em}{}\selectfont{}983, 1973 \\
1647 & 8 & 8 & 7 & 7 & 7 & 7 & $ \cup $ & $ 1646 \ (\cong 319) $ & \fontsize{0.2cm}{1em}{}\selectfont{}1782 \\
1723 & 8 & 8 & 7 & 7 & 7 & 7 & $ \cup $ & $ 1211 \ (\cong 319) $ & \fontsize{0.2cm}{1em}{}\selectfont{}1913 \\
879 & 8 & 9 & 7 & 7 & 7 & 7 & $ \cup $ & $ 847 \ (\cong 319) $ & \fontsize{0.2cm}{1em}{}\selectfont{}990, 1779, 1971, 1980 \\
894 & 8 & 9 & 7 & 7 & 7 & 7 & $ \cup $ & $ 892 \ (\cong 319) $ & \fontsize{0.2cm}{1em}{}\selectfont{}987, 1662, 2019, 2022, 5742 \\
1659 & 8 & 9 & 8 & 8 & 8 & 8 & $ \Delta $ & - & \fontsize{0.2cm}{1em}{}\selectfont{}1719, 1974 \\
5739 & 8 & 9 & 8 & 8 & 8 & 8 & $ \Delta $ & - & \fontsize{0.2cm}{1em}{}\selectfont{}- \\
5783 & 8 & 9 & 7 & 7 & 7 & 7 & $ \cup $ & $ 5782 \ (\cong 319) $ & \fontsize{0.2cm}{1em}{}\selectfont{}- \\
5790 & 8 & 9 & 7 & 7 & 7 & 7 & $ \cup $ & $ 5782 \ (\cong 319) $ & \fontsize{0.2cm}{1em}{}\selectfont{}- \\
6038 & 8 & 10 & 7 & 7 & 7 & 7 & $ \cup $ & $ 5910 \ (\cong 319) $ & \fontsize{0.2cm}{1em}{}\selectfont{}- \\
6041 & 8 & 10 & 7 & 7 & 7 & 7 & $ \cup $ & $ 5529 \ (\cong 319) $ & \fontsize{0.2cm}{1em}{}\selectfont{}- \\
1695 & 8 & 11 & 8 & 8 & 8 & 8 & $ \Delta $ & - & \fontsize{0.2cm}{1em}{}\selectfont{}1785 \\
1725 & 8 & 11 & 8 & 8 & 8 & 8 & $ \Delta $ & - & \fontsize{0.2cm}{1em}{}\selectfont{}2025 \\
5787 & 8 & 11 & 8 & 8 & 8 & 8 & $ \Delta $ & - & \fontsize{0.2cm}{1em}{}\selectfont{}- \\
5803 & 8 & 11 & 8 & 8 & 8 & 8 & $ \Delta $ & - & \fontsize{0.2cm}{1em}{}\selectfont{}- \\
6023 & 8 & 11 & 7 & 7 & 7 & 7 & $ \cup $ & $ 5895 \ (\cong 319) $ & \fontsize{0.2cm}{1em}{}\selectfont{}- \\
6630 & 8 & 11 & 7 & 7 & 7 & 7 & $ \cup $ & $ 6502 \ (\cong 319) $ & \fontsize{0.2cm}{1em}{}\selectfont{}- \\
5806 & 8 & 12 & 8 & 8 & 8 & 8 & $ \Delta $ & - & \fontsize{0.2cm}{1em}{}\selectfont{}6042 \\
5820 & 8 & 12 & 7 & 7 & 7 & 7 & $ \cup $ & $ 5692 \ (\cong 319) $ & \fontsize{0.2cm}{1em}{}\selectfont{}6634 \\
6027 & 8 & 12 & 8 & 8 & 8 & 8 & $ \Delta $ & - & \fontsize{0.2cm}{1em}{}\selectfont{}- \\
6641 & 8 & 12 & 8 & 8 & 8 & 8 & $ \Delta $ & - & \fontsize{0.2cm}{1em}{}\selectfont{}- \\
7140 & 8 & 12 & 8 & 8 & 8 & 8 & $ \Delta $ & - & \fontsize{0.2cm}{1em}{}\selectfont{}- \\
7905 & 8 & 12 & 7 & 7 & 7 & 7 & $ \cup $ & $ 7904 \ (\cong 319) $ & \fontsize{0.2cm}{1em}{}\selectfont{}- \\
5775 & 8 & 13 & 8 & 8 & 8 & 8 & $ \Delta $ & - & \fontsize{0.2cm}{1em}{}\selectfont{}- \\
5805 & 8 & 13 & 8 & 8 & 8 & 8 & $ \Delta $ & - & \fontsize{0.2cm}{1em}{}\selectfont{}- \\
6030 & 8 & 13 & 8 & 8 & 8 & 8 & $ \Delta $ & - & \fontsize{0.2cm}{1em}{}\selectfont{}- \\
6057 & 8 & 13 & 8 & 8 & 8 & 8 & $ \Delta $ & - & \fontsize{0.2cm}{1em}{}\selectfont{}- \\
6627 & 8 & 13 & 8 & 8 & 8 & 8 & $ \Delta $ & - & \fontsize{0.2cm}{1em}{}\selectfont{}- \\
6633 & 8 & 13 & 8 & 8 & 8 & 8 & $ \Delta $ & - & \fontsize{0.2cm}{1em}{}\selectfont{}- \\
5866 & 8 & 14 & 8 & 8 & 8 & 8 & $ \Delta $ & - & \fontsize{0.2cm}{1em}{}\selectfont{}6060, 6648 \\
5865 & 8 & 15 & 8 & 8 & 8 & 8 & $ \Delta $ & - & \fontsize{0.2cm}{1em}{}\selectfont{}- \\
6375 & 8 & 15 & 8 & 8 & 8 & 8 & $ \Delta $ & - & \fontsize{0.2cm}{1em}{}\selectfont{}- \\
6120 & 8 & 16 & 8 & 8 & 8 & 8 & $ \Delta $ & - & \fontsize{0.2cm}{1em}{}\selectfont{}7128, 27030 \\
1911 & 9 & 6 & 6 & 6 & 6 & 6 & - & - & \fontsize{0.2cm}{1em}{}\selectfont{}- \\
511 & 9 & 7 & 7 & 7 & 7 & 7 & - & - & \fontsize{0.2cm}{1em}{}\selectfont{}4081 \\
895 & 9 & 8 & 7 & 7 & 7 & 7 & $ \cup $ & $ 831 \ (\cong 831) $ & \fontsize{0.2cm}{1em}{}\selectfont{}991, 1021, 2035, 2042 \\
1915 & 9 & 8 & 7 & 8 & 8 & 8 & - & - & \fontsize{0.2cm}{1em}{}\selectfont{}1975 \\
1663 & 9 & 9 & 8 & 8 & 8 & 8 & $ \cup $ & $ 1662 \ (\cong 894) $ & \fontsize{0.2cm}{1em}{}\selectfont{}1783, 2038 \\
1918 & 9 & 9 & 7 & 8 & 8 & 8 & $ \cup $ & $ 1916 \ (\cong 863) $ & \fontsize{0.2cm}{1em}{}\selectfont{}2023, 5758 \\
5743 & 9 & 9 & 8 & 8 & 8 & 8 & $ \cup $ & $ 5742 \ (\cong 894) $ & \fontsize{0.2cm}{1em}{}\selectfont{}- \\
6039 & 9 & 9 & 7 & 7 & 7 & 7 & $ \cup $ & $ 5911 \ (\cong 831) $ & \fontsize{0.2cm}{1em}{}\selectfont{}- \\
15559 & 9 & 9 & 7 & 7 & 7 & 7 & $ \cup $ & $ 15555 \ (\cong 831) $ & \fontsize{0.2cm}{1em}{}\selectfont{}- \\
1727 & 9 & 10 & 8 & 8 & 8 & 8 & $ \cup $ & $ 1723 \ (\cong 1723) $ & \fontsize{0.2cm}{1em}{}\selectfont{}1787, 2041 \\
1981 & 9 & 10 & 8 & 8 & 8 & 8 & $ \cup $ & $ 1980 \ (\cong 879) $ & \fontsize{0.2cm}{1em}{}\selectfont{}2027 \\
6638 & 9 & 10 & 7 & 7 & 7 & 7 & $ \cup $ & $ 4590 \ (\cong 831) $ & \fontsize{0.2cm}{1em}{}\selectfont{}- \\
5791 & 9 & 11 & 8 & 8 & 8 & 8 & $ \cup $ & $ 5790 \ (\cong 5790) $ & \fontsize{0.2cm}{1em}{}\selectfont{}- \\
5822 & 9 & 11 & 8 & 8 & 8 & 8 & $ \cup $ & $ 5820 \ (\cong 5820) $ & \fontsize{0.2cm}{1em}{}\selectfont{}- \\
6043 & 9 & 11 & 8 & 8 & 8 & 8 & $ \cup $ & $ 6041 \ (\cong 6041) $ & \fontsize{0.2cm}{1em}{}\selectfont{}- \\
7921 & 9 & 11 & 8 & 8 & 8 & 8 & $ \cup $ & $ 7920 \ (\cong 383) $ & \fontsize{0.2cm}{1em}{}\selectfont{}- \\
5807 & 9 & 12 & 8 & 8 & 8 & 8 & $ \cup $ & $ 5295 \ (\cong 879) $ & \fontsize{0.2cm}{1em}{}\selectfont{}- \\
6046 & 9 & 12 & 8 & 8 & 8 & 8 & $ \cup $ & $ 6038 \ (\cong 6038) $ & \fontsize{0.2cm}{1em}{}\selectfont{}- \\
6059 & 9 & 12 & 8 & 8 & 8 & 8 & $ \cup $ & $ 6058 \ (\cong 383) $ & \fontsize{0.2cm}{1em}{}\selectfont{}- \\
6643 & 9 & 12 & 8 & 8 & 8 & 8 & $ \cup $ & $ 6579 \ (\cong 879) $ & \fontsize{0.2cm}{1em}{}\selectfont{}- \\
6649 & 9 & 12 & 8 & 8 & 8 & 8 & $ \cup $ & $ 6617 \ (\cong 383) $ & \fontsize{0.2cm}{1em}{}\selectfont{}- \\
7141 & 9 & 12 & 8 & 8 & 8 & 8 & $ \cup $ & $ 7077 \ (\cong 5790) $ & \fontsize{0.2cm}{1em}{}\selectfont{}- \\
7148 & 9 & 12 & 8 & 8 & 8 & 8 & $ \cup $ & $ 7116 \ (\cong 383) $ & \fontsize{0.2cm}{1em}{}\selectfont{}- \\
7907 & 9 & 12 & 8 & 8 & 8 & 8 & $ \cup $ & $ 7906 \ (\cong 5820) $ & \fontsize{0.2cm}{1em}{}\selectfont{}- \\
5821 & 9 & 13 & 8 & 8 & 8 & 8 & $ \cup $ & $ 5820 \ (\cong 5820) $ & \fontsize{0.2cm}{1em}{}\selectfont{}- \\
5870 & 9 & 13 & 8 & 8 & 8 & 8 & $ \cup $ & $ 5862 \ (\cong 5820) $ & \fontsize{0.2cm}{1em}{}\selectfont{}6076, 6650 \\
6031 & 9 & 13 & 8 & 8 & 8 & 8 & $ \cup $ & $ 6023 \ (\cong 6023) $ & \fontsize{0.2cm}{1em}{}\selectfont{}- \\
6061 & 9 & 13 & 8 & 8 & 8 & 8 & $ \cup $ & $ 6053 \ (\cong 6041) $ & \fontsize{0.2cm}{1em}{}\selectfont{}- \\
6062 & 9 & 13 & 8 & 8 & 8 & 8 & $ \cup $ & $ 6058 \ (\cong 383) $ & \fontsize{0.2cm}{1em}{}\selectfont{}- \\
6631 & 9 & 13 & 8 & 8 & 8 & 8 & $ \cup $ & $ 6630 \ (\cong 6630) $ & \fontsize{0.2cm}{1em}{}\selectfont{}- \\
6635 & 9 & 13 & 8 & 8 & 8 & 8 & $ \cup $ & $ 6634 \ (\cong 5820) $ & \fontsize{0.2cm}{1em}{}\selectfont{}- \\
6646 & 9 & 13 & 8 & 8 & 8 & 8 & $ \cup $ & $ 6630 \ (\cong 6630) $ & \fontsize{0.2cm}{1em}{}\selectfont{}- \\
7910 & 9 & 13 & 8 & 8 & 8 & 8 & $ \cup $ & $ 7908 \ (\cong 5820) $ & \fontsize{0.2cm}{1em}{}\selectfont{}- \\
5867 & 9 & 14 & 8 & 9 & 9 & 9 & $ \Delta $ & - & \fontsize{0.2cm}{1em}{}\selectfont{}- \\
6122 & 9 & 14 & 8 & 8 & 8 & 8 & $ \cup $ & $ 6058 \ (\cong 383) $ & \fontsize{0.2cm}{1em}{}\selectfont{}7129 \\
6383 & 9 & 14 & 8 & 8 & 8 & 8 & $ \cup $ & $ 6382 \ (\cong 6023) $ & \fontsize{0.2cm}{1em}{}\selectfont{}- \\
7126 & 9 & 14 & 9 & 9 & 9 & 9 & $ \Delta $ & - & \fontsize{0.2cm}{1em}{}\selectfont{}- \\
7913 & 9 & 14 & 8 & 8 & 8 & 8 & $ \cup $ & $ 7912 \ (\cong 894) $ & \fontsize{0.2cm}{1em}{}\selectfont{}- \\
6121 & 9 & 15 & 8 & 9 & 9 & 9 & $ \Delta $ & - & \fontsize{0.2cm}{1em}{}\selectfont{}- \\
27031 & 9 & 15 & 8 & 9 & 9 & 9 & $ \Delta $ & - & \fontsize{0.2cm}{1em}{}\selectfont{}- \\
1023 & 10 & 7 & 7 & 7 & 7 & 7 & - & - & \fontsize{0.2cm}{1em}{}\selectfont{}4083 \\
1919 & 10 & 8 & 7 & 7 & 7 & 7 & $ \cup $ & $ 1911 \ (\cong 1911) $ & \fontsize{0.2cm}{1em}{}\selectfont{}2039 \\
15567 & 10 & 8 & 7 & 7 & 7 & 7 & $ \div $ & $ 5189 \ (\cong 107) $ & \fontsize{0.2cm}{1em}{}\selectfont{}- \\
1791 & 10 & 9 & 8 & 8 & 8 & 8 & $ \cup $ & $ 1279 \ (\cong 511) $ & \fontsize{0.2cm}{1em}{}\selectfont{}4086 \\
1983 & 10 & 9 & 8 & 8 & 8 & 8 & $ \cup $ & $ 1967 \ (\cong 895) $ & \fontsize{0.2cm}{1em}{}\selectfont{}2043 \\
2031 & 10 & 9 & 8 & 8 & 8 & 8 & $ \cup $ & $ 1999 \ (\cong 895) $ & \fontsize{0.2cm}{1em}{}\selectfont{}2046 \\
5759 & 10 & 9 & 8 & 9 & 9 & 9 & - & - & \fontsize{0.2cm}{1em}{}\selectfont{}- \\
6014 & 10 & 10 & 7 & 9 & 9 & 9 & $ \cup $ & $ 6012 \ (\cong 1918) $ & \fontsize{0.2cm}{1em}{}\selectfont{}- \\
8177 & 10 & 10 & 8 & 8 & 8 & 8 & $ \cup $ & $ 8176 \ (\cong 511) $ & \fontsize{0.2cm}{1em}{}\selectfont{}- \\
6047 & 10 & 11 & 8 & 8 & 8 & 8 & $ \cup $ & $ 6039 \ (\cong 6039) $ & \fontsize{0.2cm}{1em}{}\selectfont{}- \\
7150 & 10 & 11 & 8 & 8 & 8 & 8 & $ \cup $ & $ 6638 \ (\cong 6638) $ & \fontsize{0.2cm}{1em}{}\selectfont{}- \\
7923 & 10 & 11 & 8 & 8 & 8 & 8 & $ \cup $ & $ 7411 \ (\cong 6638) $ & \fontsize{0.2cm}{1em}{}\selectfont{}- \\
8178 & 10 & 11 & 8 & 8 & 8 & 8 & $ \cup $ & $ 8176 \ (\cong 511) $ & \fontsize{0.2cm}{1em}{}\selectfont{}- \\
15575 & 10 & 11 & 8 & 8 & 8 & 8 & $ \cup $ & $ 15571 \ (\cong 15559) $ & \fontsize{0.2cm}{1em}{}\selectfont{}- \\
15579 & 10 & 11 & 8 & 8 & 8 & 8 & $ \cup $ & $ 15571 \ (\cong 15559) $ & \fontsize{0.2cm}{1em}{}\selectfont{}- \\
5823 & 10 & 12 & 8 & 9 & 9 & 9 & $ \cup $ & $ 5822 \ (\cong 5822) $ & \fontsize{0.2cm}{1em}{}\selectfont{}- \\
5886 & 10 & 12 & 8 & 9 & 9 & 9 & $ \cup $ & $ 5884 \ (\cong 5870) $ & \fontsize{0.2cm}{1em}{}\selectfont{}- \\
6063 & 10 & 12 & 8 & 8 & 8 & 8 & $ \cup $ & $ 5551 \ (\cong 895) $ & \fontsize{0.2cm}{1em}{}\selectfont{}- \\
6639 & 10 & 12 & 8 & 8 & 8 & 8 & $ \cup $ & $ 6638 \ (\cong 6638) $ & \fontsize{0.2cm}{1em}{}\selectfont{}- \\
6647 & 10 & 12 & 8 & 8 & 8 & 8 & $ \cup $ & $ 6519 \ (\cong 6039) $ & \fontsize{0.2cm}{1em}{}\selectfont{}- \\
6651 & 10 & 12 & 8 & 8 & 8 & 8 & $ \cup $ & $ 6587 \ (\cong 895) $ & \fontsize{0.2cm}{1em}{}\selectfont{}- \\
6654 & 10 & 12 & 8 & 8 & 8 & 8 & $ \cup $ & $ 6638 \ (\cong 6638) $ & \fontsize{0.2cm}{1em}{}\selectfont{}- \\
7164 & 10 & 12 & 8 & 8 & 8 & 8 & $ \cup $ & $ 5116 \ (\cong 895) $ & \fontsize{0.2cm}{1em}{}\selectfont{}7930 \\
7918 & 10 & 12 & 7 & 7 & 7 & 7 & $ \cup $ & $ 3822 \ (\cong 1911) $ & \fontsize{0.2cm}{1em}{}\selectfont{}- \\
7926 & 10 & 12 & 8 & 9 & 9 & 9 & $ \cup $ & $ 7924 \ (\cong 7148) $ & \fontsize{0.2cm}{1em}{}\selectfont{}- \\
8184 & 10 & 12 & 8 & 8 & 8 & 8 & $ \cup $ & $ 8176 \ (\cong 511) $ & \fontsize{0.2cm}{1em}{}\selectfont{}- \\
15834 & 10 & 12 & 8 & 9 & 9 & 9 & $ \cup $ & $ 15832 \ (\cong 5870) $ & \fontsize{0.2cm}{1em}{}\selectfont{}- \\
5871 & 10 & 13 & 8 & 9 & 9 & 9 & $ \cup $ & $ 5870 \ (\cong 5870) $ & \fontsize{0.2cm}{1em}{}\selectfont{}- \\
6077 & 10 & 13 & 8 & 9 & 9 & 9 & $ \cup $ & $ 6076 \ (\cong 5870) $ & \fontsize{0.2cm}{1em}{}\selectfont{}- \\
6078 & 10 & 13 & 8 & 9 & 9 & 9 & $ \cup $ & $ 6076 \ (\cong 5870) $ & \fontsize{0.2cm}{1em}{}\selectfont{}- \\
6126 & 10 & 13 & 8 & 8 & 8 & 8 & $ \cup $ & $ 5614 \ (\cong 895) $ & \fontsize{0.2cm}{1em}{}\selectfont{}7131 \\
6399 & 10 & 13 & 8 & 8 & 8 & 8 & $ \cup $ & $ 4351 \ (\cong 511) $ & \fontsize{0.2cm}{1em}{}\selectfont{}- \\
7127 & 10 & 13 & 8 & 9 & 9 & 9 & $ \cup $ & $ 7125 \ (\cong 5870) $ & \fontsize{0.2cm}{1em}{}\selectfont{}- \\
7134 & 10 & 13 & 9 & 9 & 9 & 9 & $ \cup $ & $ 7132 \ (\cong 6122) $ & \fontsize{0.2cm}{1em}{}\selectfont{}- \\
7143 & 10 & 13 & 8 & 9 & 9 & 9 & $ \cup $ & $ 7142 \ (\cong 7141) $ & \fontsize{0.2cm}{1em}{}\selectfont{}- \\
7149 & 10 & 13 & 8 & 9 & 9 & 9 & $ \cup $ & $ 7148 \ (\cong 7148) $ & \fontsize{0.2cm}{1em}{}\selectfont{}- \\
7911 & 10 & 13 & 8 & 9 & 9 & 9 & $ \cup $ & $ 7910 \ (\cong 7910) $ & \fontsize{0.2cm}{1em}{}\selectfont{}- \\
7915 & 10 & 13 & 8 & 9 & 9 & 9 & $ \cup $ & $ 7914 \ (\cong 5870) $ & \fontsize{0.2cm}{1em}{}\selectfont{}- \\
7929 & 10 & 13 & 8 & 9 & 9 & 9 & $ \cup $ & $ 7928 \ (\cong 6122) $ & \fontsize{0.2cm}{1em}{}\selectfont{}- \\
15830 & 10 & 13 & 8 & 9 & 9 & 9 & $ \cup $ & $ 15828 \ (\cong 1918) $ & \fontsize{0.2cm}{1em}{}\selectfont{}- \\
6123 & 10 & 14 & 8 & 9 & 9 & 9 & $ \cup $ & $ 6122 \ (\cong 6122) $ & \fontsize{0.2cm}{1em}{}\selectfont{}- \\
27039 & 10 & 14 & 8 & 9 & 9 & 9 & $ \cup $ & $ 27037 \ (\cong 7913) $ & \fontsize{0.2cm}{1em}{}\selectfont{}- \\
27606 & 10 & 14 & 8 & 9 & 10 & 10 & $ \Delta $ & - & \fontsize{0.2cm}{1em}{}\selectfont{}- \\
2047 & 11 & 8 & 8 & 8 & 8 & 8 & - & - & \fontsize{0.2cm}{1em}{}\selectfont{}4087 \\
6015 & 11 & 9 & 8 & 8 & 8 & 8 & $ \cup $ & $ 6007 \ (\cong 1919) $ & \fontsize{0.2cm}{1em}{}\selectfont{}- \\
8179 & 11 & 10 & 8 & 8 & 8 & 8 & $ \cup $ & $ 4083 \ (\cong 1023) $ & \fontsize{0.2cm}{1em}{}\selectfont{}- \\
15583 & 11 & 10 & 8 & 8 & 8 & 8 & $ \cup $ & $ 15567 \ (\cong 15567) $ & \fontsize{0.2cm}{1em}{}\selectfont{}- \\
6655 & 11 & 11 & 8 & 8 & 8 & 8 & $ \cup $ & $ 4607 \ (\cong 1023) $ & \fontsize{0.2cm}{1em}{}\selectfont{}- \\
7934 & 11 & 11 & 8 & 8 & 8 & 8 & $ \cup $ & $ 7918 \ (\cong 7918) $ & \fontsize{0.2cm}{1em}{}\selectfont{}- \\
8186 & 11 & 11 & 8 & 8 & 8 & 8 & $ \cup $ & $ 4090 \ (\cong 1023) $ & \fontsize{0.2cm}{1em}{}\selectfont{}- \\
15853 & 11 & 11 & 8 & 9 & 9 & 9 & $ \cup $ & $ 15852 \ (\cong 7164) $ & \fontsize{0.2cm}{1em}{}\selectfont{}- \\
5887 & 11 & 12 & 8 & 9 & 9 & 9 & $ \cup $ & $ 5375 \ (\cong 1791) $ & \fontsize{0.2cm}{1em}{}\selectfont{}- \\
6079 & 11 & 12 & 8 & 8 & 8 & 8 & $ \cup $ & $ 5951 \ (\cong 1919) $ & \fontsize{0.2cm}{1em}{}\selectfont{}- \\
7135 & 11 & 12 & 8 & 8 & 8 & 8 & $ \cup $ & $ 7007 \ (\cong 1919) $ & \fontsize{0.2cm}{1em}{}\selectfont{}- \\
7151 & 11 & 12 & 8 & 9 & 9 & 9 & $ \cup $ & $ 7150 \ (\cong 7150) $ & \fontsize{0.2cm}{1em}{}\selectfont{}- \\
7165 & 11 & 12 & 8 & 9 & 9 & 9 & $ \cup $ & $ 7164 \ (\cong 7164) $ & \fontsize{0.2cm}{1em}{}\selectfont{}- \\
7919 & 11 & 12 & 8 & 8 & 8 & 8 & $ \cup $ & $ 7918 \ (\cong 7918) $ & \fontsize{0.2cm}{1em}{}\selectfont{}- \\
7927 & 11 & 12 & 8 & 9 & 9 & 9 & $ \cup $ & $ 7925 \ (\cong 7923) $ & \fontsize{0.2cm}{1em}{}\selectfont{}- \\
7931 & 11 & 12 & 8 & 9 & 9 & 9 & $ \cup $ & $ 7930 \ (\cong 7164) $ & \fontsize{0.2cm}{1em}{}\selectfont{}- \\
8182 & 11 & 12 & 8 & 9 & 9 & 9 & $ \cup $ & $ 8180 \ (\cong 8178) $ & \fontsize{0.2cm}{1em}{}\selectfont{}- \\
8185 & 11 & 12 & 8 & 9 & 9 & 9 & $ \cup $ & $ 8184 \ (\cong 8184) $ & \fontsize{0.2cm}{1em}{}\selectfont{}- \\
15831 & 11 & 12 & 8 & 9 & 9 & 9 & $ \cup $ & $ 15827 \ (\cong 15575) $ & \fontsize{0.2cm}{1em}{}\selectfont{}- \\
15835 & 11 & 12 & 8 & 9 & 9 & 9 & $ \cup $ & $ 15827 \ (\cong 15575) $ & \fontsize{0.2cm}{1em}{}\selectfont{}- \\
15838 & 11 & 12 & 8 & 9 & 9 & 9 & $ \cup $ & $ 15836 \ (\cong 6126) $ & \fontsize{0.2cm}{1em}{}\selectfont{}- \\
6127 & 11 & 13 & 8 & 9 & 9 & 9 & $ \cup $ & $ 6126 \ (\cong 6126) $ & \fontsize{0.2cm}{1em}{}\selectfont{}- \\
6142 & 11 & 13 & 8 & 9 & 9 & 9 & $ \cup $ & $ 6140 \ (\cong 6126) $ & \fontsize{0.2cm}{1em}{}\selectfont{}- \\
27071 & 11 & 13 & 8 & 9 & 10 & 10 & $ \cup $ & $ 27070 \ (\cong 27039) $ & \fontsize{0.2cm}{1em}{}\selectfont{}- \\
27581 & 11 & 13 & 8 & 10 & 10 & 10 & $ \cup $ & $ 27580 \ (\cong 15830) $ & \fontsize{0.2cm}{1em}{}\selectfont{}- \\
27607 & 11 & 13 & 8 & 9 & 10 & 10 & $ \cup $ & $ 27605 \ (\cong 7929) $ & \fontsize{0.2cm}{1em}{}\selectfont{}- \\
4095 & 12 & 7 & 7 & 7 & 7 & 7 & - & - & \fontsize{0.2cm}{1em}{}\selectfont{}- \\
15615 & 12 & 9 & 8 & 8 & 8 & 8 & $ \div $ & $ 5205 \ (\cong 111) $ & \fontsize{0.2cm}{1em}{}\selectfont{}- \\
15869 & 12 & 10 & 8 & 8 & 8 & 8 & $ \phantom{{}^*}\div^* $ & $ 12785 \ (\cong 863) $ & \fontsize{0.2cm}{1em}{}\selectfont{}- \\
16380 & 12 & 10 & 8 & 8 & 8 & 8 & $ \div $ & $ 5460 \ (\cong 126) $ & \fontsize{0.2cm}{1em}{}\selectfont{}- \\
7167 & 12 & 11 & 8 & 8 & 8 & 8 & $ \downarrow $ & $ 6604 \ (\cong 319) $ & \fontsize{0.2cm}{1em}{}\selectfont{}- \\
7935 & 12 & 11 & 8 & 9 & 9 & 9 & $ \cup $ & $ 7934 \ (\cong 7934) $ & \fontsize{0.2cm}{1em}{}\selectfont{}- \\
8183 & 12 & 11 & 8 & 9 & 9 & 9 & $ \cup $ & $ 8181 \ (\cong 8179) $ & \fontsize{0.2cm}{1em}{}\selectfont{}- \\
8187 & 12 & 11 & 8 & 9 & 9 & 9 & $ \cup $ & $ 8186 \ (\cong 8186) $ & \fontsize{0.2cm}{1em}{}\selectfont{}- \\
8190 & 12 & 11 & 8 & 9 & 9 & 9 & $ \cup $ & $ 8188 \ (\cong 8186) $ & \fontsize{0.2cm}{1em}{}\selectfont{}- \\
15839 & 12 & 11 & 8 & 9 & 9 & 9 & $ \cup $ & $ 15837 \ (\cong 7934) $ & \fontsize{0.2cm}{1em}{}\selectfont{}- \\
15855 & 12 & 11 & 8 & 9 & 9 & 9 & $ \cup $ & $ 15823 \ (\cong 15583) $ & \fontsize{0.2cm}{1em}{}\selectfont{}- \\
15870 & 12 & 11 & 8 & 9 & 9 & 9 & $ \cup $ & $ 15868 \ (\cong 8186) $ & \fontsize{0.2cm}{1em}{}\selectfont{}- \\
6143 & 12 & 12 & 8 & 9 & 9 & 9 & $ \cup $ & $ 6135 \ (\cong 6079) $ & \fontsize{0.2cm}{1em}{}\selectfont{}- \\
27135 & 12 & 12 & 8 & 9 & 9 & 9 & $ \phantom{{}^*}\div^* $ & $ 27135 \ (\cong 27135) $ & \fontsize{0.2cm}{1em}{}\selectfont{}- \\
27583 & 12 & 12 & 8 & 9 & 9 & 9 & $ \cup $ & $ 27579 \ (\cong 7919) $ & \fontsize{0.2cm}{1em}{}\selectfont{}- \\
27615 & 12 & 12 & 8 & 9 & 10 & 10 & $ \cup $ & $ 27613 \ (\cong 15838) $ & \fontsize{0.2cm}{1em}{}\selectfont{}- \\
27645 & 12 & 12 & 8 & 10 & 10 & 10 & $ \cup $ & $ 27644 \ (\cong 15838) $ & \fontsize{0.2cm}{1em}{}\selectfont{}- \\
28662 & 12 & 12 & 8 & 9 & 9 & 9 & $ \div $ & $ 19924 \ (\cong 1647) $ & \fontsize{0.2cm}{1em}{}\selectfont{}- \\
28665 & 12 & 12 & 8 & 10 & 10 & 10 & $ \cup $ & $ 28664 \ (\cong 8182) $ & \fontsize{0.2cm}{1em}{}\selectfont{}- \\
8191 & 13 & 10 & 8 & 8 & 8 & 8 & $ \cup $ & $ 4095 \ (\cong 4095) $ & \fontsize{0.2cm}{1em}{}\selectfont{}- \\
15871 & 13 & 10 & 8 & 9 & 9 & 9 & $ \div $ & $ 12787 \ (\cong 895) $ & \fontsize{0.2cm}{1em}{}\selectfont{}- \\
16381 & 13 & 10 & 8 & 9 & 9 & 9 & $ \div $ & $ 13297 \ (\cong 895) $ & \fontsize{0.2cm}{1em}{}\selectfont{}- \\
27647 & 13 & 11 & 8 & 9 & 10 & 10 & $ \cup $ & $ 27643 \ (\cong 15839) $ & \fontsize{0.2cm}{1em}{}\selectfont{}- \\
28663 & 13 & 11 & 8 & 9 & 10 & 10 & $ \cup $ & $ 28661 \ (\cong 15870) $ & \fontsize{0.2cm}{1em}{}\selectfont{}- \\
28667 & 13 & 11 & 8 & 10 & 10 & 10 & $ \cup $ & $ 28666 \ (\cong 15870) $ & \fontsize{0.2cm}{1em}{}\selectfont{}- \\
16383 & 14 & 9 & 8 & 8 & 8 & 8 & $ \div $ & $ 5461 \ (\cong 127) $ & \fontsize{0.2cm}{1em}{}\selectfont{}- \\
28671 & 14 & 10 & 8 & 9 & 9 & 9 & $ \cup $ & $ 20479 \ (\cong 8191) $ & \fontsize{0.2cm}{1em}{}\selectfont{}- \\
32511 & 14 & 10 & 8 & 10 & 10 & 10 & - & - & \fontsize{0.2cm}{1em}{}\selectfont{}- \\
32766 & 14 & 10 & 8 & 9 & 10 & 10 & - & - & \fontsize{0.2cm}{1em}{}\selectfont{}- \\
32767 & 15 & 9 & 8 & 9 & 9 & 9 & - & - & \fontsize{0.2cm}{1em}{}\selectfont{}- \\
65535 & 16 & 8 & 8 & 8 & 8 & 8 & - & - & \fontsize{0.2cm}{1em}{}\selectfont{}- \\ 
    \end{supertabular}
}
 
\newpage

\section{Outlook}
\noindent
We computed the extension complexities of all $ 0/1 $-polytopes up to dimension $ 4 $ by providing minimum size
extensions and matching lower bounds.
In particular, all implicitly computed minimum size extensions were induced by simple geometric operations and satisfy
strong properties:
\begin{observation*}
    For every $ 0/1 $-polytope of dimension at most $ 4 $ there exists a minimum size extension that is even a nice $
    0/1 $-extension.
\end{observation*}
\noindent
As mentioned throughout this paper, the authors are not aware of any arguments that rule out the existence of such
extensions for general $ 0/1 $-polytopes.
On the other hand, there is no theoretical evidence that minimum size extensions with such strong properties should
always exist.
Thus, any new insights regarding such properties of minimum size extensions are of certain interest.

On the computational side, as a next natural step one might consider computations in dimension $ 5 $.
However, it seems to be a much more difficult task to achieve a complete list of results in this case.
There are $ 1{,}226{,}525 $ different $ 0/1 $-equivalence classes~\cite{Aichholzer00}.
Thus, even the determination of all affine equivalence classes is presumably very time-consuming.
Independently, due to the sizes of the corresponding slack matrices, computing all lower bounds presented in this paper
requires sophisticated algorithms.
In addition, it is questionable whether rather simple (lower and upper) bounds as presented in this paper suffice to
determine all extension complexities.

In this paper we considered the \emph{linear} extension complexity of a polytope $ P $, which can also be seen as the
smallest $ r $ such that $ P $ can be written as a linear projection of an affine slice of the nonnegative orthant $
\R^r_+ $.
A more general way of representing polytopes that recently has become of particular interest is the concept of
semidefinite extensions.
Here, the analogous quantity is the so-called \emph{semidefinite} extension complexity, which is defined as the
smallest $ r $ such that $ P $ can be written as a linear projection of an affine slice of the cone $ \mathbb{S}^r_+ $
of positive semidefinite matrices of size $ r \times r $.
Although there are already several papers concerning this quantity, the field lacks of good upper and lower bounds on
sizes of semidefinite extensions.
For instance, Gouveia~et~al.~\cite{GouveiaRT13} examine properties of polytopes $ P $ whose semidefinite extension
complexity equals $ \dim(P) + 1 $ (which is a general lower bound~\cite{LeeT12}).
Using their observations together with our results in Table~\ref{tab:dim3} and the basic construction of semidefinite
extensions by taking the positive Hadamard square root of a slack matrix (see~\cite{GouveiaRT13}), it is easy to
determine the
semidefinite extension complexities of all $ 0/1 $-polytopes up to dimension $ 3 $.
However, in order to continue the computations for dimension $ 4 $, it seems that new -- yet unknown -- bounds have to
be taken into account.
We hope that our computations provide a helpful basis towards this task.

\bibliographystyle{plain}
\bibliography{references}

\appendix
\section*{Appendix}
\renewcommand{\thesection}{A}

\subsection{Proof of Lemma~\ref{lem:polytope}}
\label{proof:polytope}
Let $ (Q,\pi) $ be an extension for $ P \neq \emptyset $ such that $ Q \subseteq \R^q $ has smallest possible dimension
and let us write $ Q = \conv(V) + \rec(Q) $ with $ \emptyset \neq V \subseteq \R^q $ finite.
Suppose for the sake of contradiction that $ Q $ is unbounded, i.e., there exists vector $ c \in \rec(Q) \setminus \{
\zerovec \} $. Since
\[
    P = \pi(Q) = \pi(Q + \rec(Q)) = \pi(Q) + \pi(\rec(Q)) = P + \pi(\rec(Q))
\]
and $ P $ is bounded, we see that the $ \rec(Q) $ lies in the kernel of $ \pi $.

Let $ \gamma := \max \{ \langle c,v \rangle : v \in \conv(V) \} $ and define $ Q' := \{ y \in Q : \langle c,y \rangle =
\gamma \} $.
We claim that $ (Q',\pi) $ is an extension for $ P $.
Clearly, it holds that $ \pi(Q') \subseteq P $.
For any $ x \in P $ let $ y = v + w \in Q $ such that $ \pi(y) = x $ with $ v \in \conv(V) $ and $ w \in \rec(Q) $.
Setting $ \lambda := \frac{\gamma - \langle c,v \rangle}{\|c\|^2} \geq 0 $ and $ y' := v + \lambda c \in Q $, we have
that
\[
    \langle c,y' \rangle = \langle c,v + \lambda c \rangle
    = \langle c,v \rangle + \langle c,\frac{\gamma - \langle c,v \rangle}{\|c\|^2} \cdotp c \rangle
    = \langle c,v \rangle + \gamma - \langle c,v \rangle
    = \gamma
\]
and hence $ y' \in Q' $. Moreover, we obtain
\[
    \pi(y') = \pi(v) + \lambda \pi(c) = \pi(v) = \pi(v) + \pi(w) = \pi(y) = x,
\]
which implies $ x \in \pi(Q') $ and thus $ \pi(Q') = P $ holds indeed.

Observing that $ \dim(Q') < \dim(Q) $ and that $ Q' $ has at most as many facets as $ Q $ yields the desired
contradiction. \qed

\end{document}